\documentclass[twoside,centertags]{amsart}

\usepackage[cmtip,color,all]{xy}
\usepackage{amsmath}
\usepackage{amsfonts}
\usepackage{amssymb}
\usepackage{amsthm}
\usepackage{graphicx}
\usepackage{mathrsfs}
\usepackage{hyperref}

\newtheorem{theorem}{Theorem}[section]
\newtheorem*{theoremA}{Theorem}

\newtheorem{lemma}[theorem]{Lemma}
\newtheorem{proposition}[theorem]{Proposition}
\theoremstyle{definition}
\newtheorem{definition}[theorem]{Definition}
\newtheorem{remark}[theorem]{Remark}
\newtheorem{addendum}[theorem]{Addendum}
\newtheorem{example}[theorem]{Example}

\newcommand{\C}{\mathscr{C}}
\newcommand{\cab}{\mathcal{C}}
\renewcommand{\a}{\mathcal{A}}
\newcommand{\A}{\mathscr{A}}
\newcommand{\E}{\mathscr{E}}

\newcommand{\B}{\mathscr{B}}
\newcommand{\D}{\mathscr{D}}
\newcommand{\s}{\mathrm{S}}
\newcommand{\sgl}{\mathrm{\widehat{S}}}

\newcommand{\Ch}{\mathrm{Ch^b}}
\newcommand{\iso}{\mathrm{iso}}
\newcommand{\colim}{\mathrm{colim}}
\newcommand{\hocolim}{\mathrm{hocolim}}

\begin{document}

\title[D\'{e}vissage]{D\'{e}vissage for Waldhausen $K$-theory} %

\author[G. Raptis]{George Raptis}

\begin{abstract}
A d\'evissage--type theorem in algebraic $K$-theory is a statement that identifies the $K$-theory of a Waldhausen category $\C$ in terms of the $K$-theories of a collection of Waldhausen subcategories of $\C$ when a d\'evissage condition about the existence of appropriate finite filtrations is satisfied.  We distinguish between d\'evissage theorems of \emph{single type} and of \emph{multiple type} depending on the number of Waldhausen subcategories and their properties. The main representative examples of such theorems are Quillen's original d\'evissage theorem for abelian categories (single type) and Waldhausen's theorem on spherical objects for more general Waldhausen categories (multiple type). In this paper, we study some general aspects of
d\'evissage--type theorems and prove a general d\'evissage theorem of single type and a general d\'evissage theorem of multiple type.    
\end{abstract}

\address{\newline
G. Raptis \newline
Fakult\"at f\"ur Mathematik, Universit\"at Regensburg,  93040 Regensburg, Germany}
\email{georgios.raptis@ur.de}

\maketitle 

\section{Introduction}

The d\'{e}vissage theorem of Quillen \cite{Qu} is a fundamental theorem in algebraic $K$-theory with many applications. The theorem states that given a full exact inclusion $\A \hookrightarrow \C$ of abelian 
categories, where $\A$ is closed under subobjects and quotients in $\C$ and the following \emph{d\'evissage condition} is satisfied: every object $C \in \C$ 
admits a finite filtration
$$0=C_0 \subseteq C_1 \subseteq \cdots \subseteq C_{n-1} \subseteq C_n =C$$
such that $C_i / C_{i-1} \in \A$ for $i \geq 1$, then the induced map $K(\A) \xrightarrow{\simeq} K(\C)$ is a homotopy equivalence. This theorem is an important source of 
$K$-equivalences which do not arise from an equivalence between the underlying homotopy theories -- in other words, it makes essential use of the ``group completion'' process that defines algebraic $K$-theory. 

\smallskip 

Waldhausen \cite{Wa} extended the definition of Quillen $K$-theory to categories with cofibrations and weak equivalences (\emph{Waldhausen categories}) and proved generalizations of Quillen's fundamental 
theorems in this general context, but the d\'{e}vissage theorem has been a notable exception. The problem of finding a suitable 
generalization of the d\'{e}vissage theorem to Waldhausen $K$-theory was stated explicitly by Thomason--Trobaugh \cite[1.11.1]{TT} (see also Waldhausen \cite[p. 188]{Wa2}). More precisely, the problem asks for a general result in Waldhausen $K$-theory that specializes to Quillen's d\'evissage theorem when applied to the categories of bounded chain complexes. 

On the other hand, Waldhausen's theorem on spherical objects \cite[Theorem 1.7.1]{Wa} may be considered as a d\'evissage--type theorem, even though 
it is not related to Quillen's d\'evissage theorem. The theorem applies to the case of a Waldhausen category $\C$ which satisfies certain homotopical assumptions and has an associated collection $\A = (\A_i)$ of Waldhausen subcategories of \emph{spherical objects in} $\C$ \emph{of dimension} $i$, i.e., objects in $\C$ which are concentrated in degree $i$ in a suitable sense. Under the assumption that the morphisms in $\C$ satisfy a multiple type version of the d\'evissage 
condition in Quillen's theorem, the theorem states that a canonical comparison map:
$$\underrightarrow{\hocolim}_{(\Sigma)} K(\A_i) \xrightarrow{\simeq} K(\C)$$
is a homotopy equivalence. This theorem was used by Waldhausen \cite{Wa} in order to obtain a description of the algebraic $K$-theory of a space in terms of matrices which is analogous to Quillen's algebraic $K$-theory of rings. Waldhausen's theorem can be considered as a d\'evissage--type theorem if the latter characterization is understood to refer to a statement 
that identifies the $K$-theory of a Waldhausen category $\C$ in terms of the $K$-theories of Waldhausen subcategories $\A = (\A_i),$ $\A_i \subset \C,$ under the assumption of a d\'evissage condition
about the existence of appropriate finite filtrations (or factorizations) of morphisms in $\C$. Besides the theorems of Quillen and Waldhausen, \emph{d\'evissage} in $K$-theory has 
been studied extensively in the literature, especially in recent years, and some further examples of d\'evissage--type theorems include: the  Gillet--Waldhausen theorem \cite{TT}, the d\'evissage theorem of Blumberg--Mandell \cite{BM}, Barwick's  `theorem of the heart' \cite{Bar} and its extension to non-connective $K$-theory based on vanishing results for negative $K$-theory \cite{AGH}.

\smallskip

Despite the common d\'evissage--type quality of the theorems of Quillen and Waldhausen, there are also important differences between these two theorems that are worth 
making explicit. The obvious one is that in Quillen's theorem there is a single subcategory for filtering the objects in $\C$, whereas in Waldhausen's theorem, we have for 
a similar purpose a collection of subcategories $\A = (\A_i)$ that are suitably related. Secondly, in the context of Quillen's theorem, the subcategory $\A \subset \C$ 
(or rather, $\Ch(\A) \subset \Ch(\C)$) is closed under (homotopy) pushouts, but this fails for the subcategories $\A_i \subset \C$ in Waldhausen's context. Thirdly,
in Quillen's theorem, the subcategory $\Ch(\A) \subset \Ch(\C)$ is not homotopically full in general, but this property will typically hold for the subcategories 
$\A_i \subset \C$ in Waldhausen's theorem.  For the purposes of our analysis of d\'evissage--type theorems in this paper, we find it helpful to distinguish 
between these two cases of d\'evissage--type theorems, which we will refer to as \emph{single type d\'evissage} and \emph{multiple type d\'evissage}, respectively. 
From this perspective, the d\'evissage--type theorem of Gillet--Waldhausen (see \cite{TT}), the d\'evissage theorem of Blumberg--Mandell \cite{BM} and 
Barwick's  theorem of the heart \cite{Bar} belong to the category of d\'evissage theorems of multiple type -- where, in addition, the maps $\Sigma \colon K(\A_i) \stackrel{\simeq}{\to} K(\A_{i+1})$ are actually homotopy equivalences. 

\medskip

The purpose of this paper is to study some general aspects of d\'evissage--type theorems and prove a d\'evissage theorem of single type (Theorem \ref{devissage-1}) and a d\'evissage theorem of multiple type (Theorem \ref{devissage-multiple-type}). The proofs of these theorems use a common 
general method, in the spirit of \cite{Wa}, whose application in each case is distinguished by arguments specific to each case. We will work with the 
following general version of the d\'evissage condition: given a Waldhausen category $\C$ 
and a collection of Waldhausen subcategories $\A = (\A_i)_{i \geq 1}$ (which may be constant), we say that $(\C, \A)$ satisfies the \emph{d\'{e}vissage condition} if for every morphism $f: X \rightarrow Y$ in $\C$, there is a weak equivalence 
$g: Y \stackrel{\sim}{\rightarrow} Y'$ such that the composition $gf: X \to Y'$ 
admits a factorization
  $$X = X_0 \rightarrowtail X_1 \rightarrowtail \cdots \rightarrowtail X_{m-1} \rightarrowtail X_m \stackrel{\sim}{\to} Y'$$
where $X_i / X_{i-1} \in \A_i$ for every $i \geq 1$.

The following is our main d\'evissage theorem of single type. 

\begin{theoremA}[Single type d\'evissage, see Theorem \ref{devissage-1}] Let $(\C, \A)$ be a good Waldhausen pair. Suppose that $(\C, \A)$ is admissible and satisfies the d\'{e}vissage condition. Then the inclusion $\A \hookrightarrow \C$ induces a homotopy equivalence $K(\A) \stackrel{\simeq}{\rightarrow} K(\C).$
\end{theoremA}

A good Waldhausen pair $(\C, \A)$ consists of a Waldhausen category $\C$ and a Waldhausen subcategory $\A \subset \C$ which satisfy certain standard homotopical properties -- we refer to Section \ref{recollect} and Definition \ref{good-Wald-pair-def}. The technical and somewhat mysterious \emph{admissibility assumption} is studied in Section \ref{admissible-Wald-pairs} where also some examples of admissible Waldhausen pairs are shown. The relation to Quillen's theorem is discussed in Subsection \ref{abelian-example}. Moreover, we explain how Waldhausen's fundamental Additivity Theorem can be recovered from the single type d\'evissage theorem in the case of a good Waldhausen category (see Example \ref{additivity-example}).

\smallskip 

In the multiple type case, our main d\'evissage theorem is as follows. 

\begin{theoremA}[Multiple type d\'evissage, see Theorem \ref{devissage-multiple-type}] 
Let $\C$ be a good Waldhausen category equipped with a cylinder functor which satisfies the cylinder axiom. Let $\A = (\A_i)_{i \geq 1}$ be a collection of Waldhausen subcategories of $\C$ such that for every $i \geq 1$:
\begin{itemize}
\item[(a)] $(\C, \A_i)$ is a replete Waldhausen pair (Definition \ref{replete-def}), 
\item[(b)] $\A_i$ is closed under extensions in $\C$,
\item[(c)] the suspension functor $\Sigma \colon \C \to \C$ sends $\A_i$ to $\A_{i+1}$.
\end{itemize}
Suppose that $(\C, \A)$ is admissible, has the cancellation property, and satisfies the d\'evissage condition. Then the canonical map 
$$\underrightarrow{\hocolim}_{(\Sigma)} K(\A_i) \stackrel{\simeq}{\longrightarrow} \underrightarrow{\hocolim}_{(\Sigma)} K(\C) \simeq K(\C)$$
is a homotopy equivalence. 
\end{theoremA}

The admissibility assumption and the cancellation property are explained in Subsections \ref{admissibility-2-subsec} and \ref{devissage-cond-subsec-2}, respectively. This theorem generalizes Waldhausen's theorem (see Example \ref{Waldhausen-theorem}) and the Gillet--Waldhausen theorem can also be deduced (see Example \ref{Gillet-Waldhausen}). In addition, the assumptions of the theorem are satisfied in the case of a weight structure for a stable homotopy theory. Thus, the theorem can also be applied to conclude a `theorem of the heart' for weight structures, as the one obtained in \cite{Fo} and \cite[8.1.3]{HS} (see Example \ref{theorem_of_the_heart}).

\medskip

\noindent \textbf{Outline of the proofs.} Let us next give an outline of the general method of proof and describe the meaning of some of the assumptions. Let $\C$ denote a Waldhausen category and let $\A = (\A_i)_{i \geq 1}$ be a collection of Waldhausen subcategories of $\C$. 

We will consider a Waldhausen subcategory $\s_n \C_{\A}$ of $\s_n \C$ which consists of those filtered objects $X_{\bullet, \bullet} \in \s_n \C$:
$$\ast \rightarrowtail X_{0,1} \rightarrowtail \cdots \rightarrowtail X_{0,n}$$
whose successive cofibers $X_{i-1, i}$ are in $\A_i$ for $i \geq 1$. Applying Waldhausen's Additivity Theorem \cite{Wa}, we may identify the $K$-theory of this Waldhausen category $\s_n \C_{\A}$ as follows (Proposition \ref{ho-eq-1}): 
\begin{equation} \label{additivity-part}
K(\s_n \C_{\A}) \simeq \prod_1^n K(\A_i).
\end{equation}
In addition, we will consider the Waldhausen category $\sgl_n \C_{\A}$ which has the same underlying category with cofibrations as $\s_n \C_{\A}$, but the weak equivalences are detected at the underlying total object $X_{0,n}$. We may also allow $n$ to be arbitrarily large and consider the respective Waldhausen categories $\s_{\infty} \C_{\A}$ and $\sgl_{\infty} \C_{\A}$. Note that the Waldhausen category $\sgl_{\infty} \C_{\A}$ models the homotopy theory of objects in $\C$ equipped with a bounded filtration by objects whose successive cofibers are in $\A$. 

Under some general homotopical assumptions on $(\C, \A)$, there is a homotopy fiber sequence:
\begin{equation} \label{fiber-seq-part}
K(\s^w_{\infty} \C_{\A}) \to K(\s_{\infty} \C_{\A}) \to K(\sgl_{\infty} \C_{\A})
\end{equation}
where $\s^w_{\infty} \C_{\A} \subset \s_{\infty} \C_{\A}$ denotes the full Waldhausen subcategory of those objects which are weakly trivial in $\sgl_{\infty} \C_{\A}$  (see Proposition \ref{fiber seq}).  The admissibility assumption essentially identifies the $K$-theory of $\s^w_{\infty} \C_{\A}$ 
(we refer to Section \ref{admissible-Wald-pairs} for the single type case, and to Subsection \ref{admissibility-2-subsec} for the multiple type case). 

There is a forgetful exact functor 
$ev_{\infty} \colon \sgl_{\infty} \C_{\A} \to \C$ that sends a filtered object to the underlying total object. Under the assumption of the d\'evissage condition for $(\C, \A)$, we prove that the exact functor $ev_{\infty}$ induces  a homotopy equivalence:
\begin{equation} \label{devissage-part}
K(\sgl_{\infty} \C_{\A}) \xrightarrow{\simeq} K(\C).
\end{equation}
(See Proposition \ref{devissage-cond-prop} for the single type case and Proposition \ref{devissage-cond-prop2} for the multiple type case.)  Combining the above, we deduce an identification of $K(\C)$ in terms of $K(\A_i)$, for $i \geq 1$, as in the theorems of Quillen and Waldhausen respectively (see Theorem \ref{devissage-1} for the single type case and Theorem \ref{devissage-multiple-type} for the multiple type case).

\medskip

\noindent \textbf{Organization of the paper.} In Section 2, we recall some background material about Waldhausen categories and review some of the fundamental theorems of Waldhausen $K$-theory. In Section 3, we introduce the Waldhausen categories $\s_n \C_{\A}$ and $\sgl_n \C_{\A}$ and prove the homotopy 
equivalence \eqref{additivity-part} and the homotopy fiber sequence \eqref{fiber-seq-part}. In Section 4, we discuss the admissibility assumption in the single type case and give 
some examples of classes of admissible Waldhausen pairs $(\C, \A)$. 

We prove the homotopy equivalence \eqref{devissage-part} in the single type case in Section 5. As a consequence, 
we then deduce the d\'evissage theorem of single type (Theorem \ref{devissage-1}). In addition, we explain how the d\'evissage condition follows from the existence of a \emph{d\'evissage functor} (Subsection \ref{devissage-functors}). We also discuss the relation of the single type d\'evissage theorem 
to Quillen's d\'evissage theorem (Subsection \ref{abelian-example}) and to the Additivity Theorem (Example \ref{additivity-example}). 

In Section 6, we first introduce an 
abstract version of the context of Waldhausen's theorem on spherical objects (Subsection \ref{admissibility-2-subsec}) and explain the assumptions in the multiple type d\'evissage 
theorem. Then we prove the homotopy equivalence \eqref{devissage-part} in the multiple type case (Subsection \ref{devissage-cond-subsec-2}). Finally, in Subsection \ref{main-thm-2}, 
we deduce the d\'evissage theorem of multiple type (Theorem \ref{devissage-multiple-type}) and discuss its relation to Waldhausen's theorem (Example \ref{Waldhausen-theorem}), the Gillet--Waldhausen theorem (Example \ref{Gillet-Waldhausen}) and the theorem of the heart for weight structures (Example \ref{theorem_of_the_heart}).  

\medskip

\noindent \textbf{Acknowledgements.} The author is grateful for the support and the hospitality of the Hausdorff Institute for Mathematics in Bonn during the 
program ``Topology'' -- where a first outline of this work was worked out  -- and for the support and the hospitality of the Isaac Newton Institute 
in Cambridge during the research programme ``Homotopy harnessing higher structures'' -- where a final version of this work was first completed. The author 
was also partially supported by \emph{SFB 1085 --- Higher Invariants} (University of Regensburg) funded by the DFG.

\section{Recollections} \label{recollect}

In this section we fix some terminology about Waldhausen categories and recall some of the fundamental theorems of Waldhausen $K$-theory that will be needed in the proofs of our main results. Throughout the paper, we will mostly work with Waldhausen categories which have the 2-out-of-3 property and admit factorizations; under homotopical assumptions of this type, there is no essential difference between this approach and the approach via pointed (small) $\infty$-categories with finite colimits (see \cite[7.2]{BGT}). 

\subsection{Waldhausen categories} A Waldhausen category is a small category $\C$ with cofibrations $\mathrm{co}\C$ and weak equivalences $\mathrm{w} \C$ in the sense of \cite{Wa}. We recall that $\C$ has a zero object $\ast$. The cofibrations form a subcategory $\mathrm{co}\C \subset 
\C$ which contains the isomorphisms and the morphisms from $\ast$. In addition, pushouts along a cofibration exist in $\C$, and $\mathrm{co}\C$ is 
closed under these pushouts. The weak equivalences form a subcategory $\mathrm{w} \C \subset \C$ which contains the isomorphisms and the gluing axiom holds \cite{Wa}. Examples include the exact categories (in the sense 
of Quillen \cite{Qu}) and appropriate full subcategories of cofibrant objects in pointed model categories. Cofibrations will be indicated  by $\rightarrowtail$ and weak equivalences by $\stackrel{\sim}{\rightarrow}$.

We will also need to consider Waldhausen categories which have the following additional properties. 

\begin{definition} \label{Wald-cat-def}
Let $\C$ be a Waldhausen category. 
\begin{itemize}
\item[(a)] We say that $\C$ has the \emph{2-out-of-3 property} if given composable morphisms $X \xrightarrow{f} Y \xrightarrow{g} Z$ in $\C$,
then all three morphisms $f$, $g$, and $gf$, are weak equivalences, whenever any two of them are. 
\item[(b)] We say that $\C$ \emph{admits factorizations} if every morphism in $\C$ can be written as the composition 
of a cofibration followed by a weak equivalence.
\item[(c)] $\C$ is called \emph{good} if it has the 2-out-of-3 property and admits factorizations.
\end{itemize}
\end{definition} 

\begin{remark}
The 2-out-of-3 property corresponds to the \emph{saturation axiom} in \cite{Wa}. A good Waldhausen category is called \emph{derivable} in \cite{Ci}.
\end{remark}

A functor between Waldhausen categories $F: \C \to \C'$ is called exact if it preserves all the relevant structure (zero object, cofibrations, weak equivalences, and pushouts along a cofibration).

\subsection{Waldhausen subcategories} Let $\C$ be a Waldhausen category. A subcategory $\A \subset \C$ is called a Waldhausen subcategory of $\C$ 
if $\A$ becomes a Waldhausen category when equipped with the following classes of cofibrations and weak equivalences:
\begin{itemize}
\item[(i)] a morphism in $\A$ is in $\mathrm{co}\A$ if it is in $\mathrm{co}\C$ and has a cofiber in $\A$,
\item[(ii)] a morphism in $\A$ is in $\mathrm{w} \A$ if it is in $\mathrm{w} \C$,  
\end{itemize}
and also the inclusion functor $\A \hookrightarrow \C$ is exact. In other words, the zero object $* \in \C$ is also a zero object in $\A$, $\mathrm{co} \A$ is a subcategory of $\A$, and the pushout in $\C$ of a diagram in $\A$ along a cofibration (in $\C$ with cofiber in $\A$) defines also a pushout in $\A$. Note that we do not assume that $\A \subset \C$ is full in general. A pair $(\C, \A)$ where $\C$ is a Waldhausen category and $\A \subset \C$ a Waldhausen subcategory will be called a \emph{Waldhausen pair}. 

Clearly, a Waldhausen subcategory $\A \subset \C$ has the 2-out-of-3 property if $\C$ does. On the other hand,  one needs additional assumptions on $(\C, \A)$ in general for the existence of factorizations in $\A$. 

\subsection{Waldhausen $K$-theory} Waldhausen's $\s_{\bullet}$-construction associates to each Waldhausen category $\C$ a simplicial object $\s_{\bullet} \C$ 
of Waldhausen categories. Following \cite{Wa}, let $\mathrm{Ar}[n]$ denote the category of morphisms and commutative squares in the poset $[n]$, $n \geq 0$. $\s_n \C$ is the full subcategory of
the category of functors
$$X: \mathrm{Ar}[n] \to \C$$
that is spanned by the functors $X$ such that:
\begin{itemize}
 \item[(i)] $X(i \to i)$ is the zero object $* \in \C$,
 \item[(ii)] for every $0 \leq i \leq j \leq k \leq l \leq n$, the square 
 \[
  \xymatrix{
  X(i \to k) \ar[r] \ar[d] & X(i  \to l) \ar[d] \\
  X(j \to k) \ar[r] & X(j \to l)
  }
 \]
is a pushout in $\C$ and the horizontal morphisms are cofibrations in $\C$.
\end{itemize}
We will abbreviate $X(i \to j)$ to $X_{i,j}$.

Such a functor $X \colon \mathrm{Ar}[n] \to \C$ is determined up to isomorphism by the filtered object 
$$\ast = X_{0,0} \rightarrowtail X_{0,1} \rightarrowtail X_{0,2} \rightarrowtail \cdots \rightarrowtail X_{0,n}$$
since any other value of $X$ is either the zero object or a cofiber of a morphism from this filtration. We will sometimes 
refer to such functors $(X_{\bullet, \bullet})$ as \textit{staircase} diagrams. An object in $\s_2 \C$ is given by a cofiber 
sequence $(X_{0,1} \rightarrowtail X_{0,2} \twoheadrightarrow X_{1,2})$. 

\medskip

The category $\s_n \C$ is naturally a Waldhausen 
category where the weak equivalences are defined pointwise (see \cite{Wa}). The simplicial operators are defined by 
precomposition of functors and induce exact functors of Waldhausen categories. Moreover, $\C \mapsto \s_{\bullet} \C$ is a 
functor from the category of Waldhausen categories and exact functors to simplicial objects in this category. We note 
that if $\C$ is good, then so is $\s_n \C$ for any $n \geq 0$, see \cite[A.9]{Sc}, \cite{Ci}.

\medskip

The restriction to the subcategory of weak equivalences $w \s_{\bullet} \C$ yields a simplicial object in the category of
 small categories. The Waldhausen $K$-theory of $\C$ is the loop space of the geometric realization of the bisimplicial set associated to this 
simplicial category: 
$$K(\C) : = \Omega|N_{\bullet} w \s_{\bullet} \C|.$$

\subsection{The Additivity Theorem} We recall the statement of Waldhausen's fundamental Additivity 
Theorem (see \cite[1.3--1.4]{Wa}). 

\begin{theorem}[Additivity Theorem] \label{addit-thm}
Let $\C$ be a Waldhausen category. Then the exact functor 
$$(d_2, d_0): \s_2 \C \to \C \times \C, \  \ (A \rightarrowtail C \twoheadrightarrow B) \mapsto (A, B),$$ 
induces a homotopy equivalence $K(\s_2 \C) \xrightarrow{\simeq} K(\C) \times K(\C)$. 
\end{theorem}

We also recall a different version of the Additivity Theorem that will be used in later sections (see \cite[Proposition 1.3.2]{Wa}). Given a Waldhausen category $\C$ and Waldhausen subcategories $\A, \B \subset \C$, we denote by $E(\A, \C, \B)$ the Waldhausen subcategory of $\s_2 \C$ whose objects 
are cofiber sequences 
$$A \rightarrowtail C \twoheadrightarrow B, \ \text{ where } A \in \A, \ B \in \B,$$
and morphisms are morphisms of cofiber sequences 
\[
 \xymatrix{
 A \ar@{>->}[r] \ar[d]^f & C \ar@{->>}[r] \ar[d]^h & B \ar[d]^g \\
 A' \ar@{>->}[r] & C' \ar@{->>}[r] & B'
 }
\]
where $f$ is  in $\A$ and $g$ is  in $\B$.  

\begin{theorem}[Additivity Theorem, Version 2] \label{addit-thm-2}
Let $\C$ be a Waldhausen category and let $\A, \B \subset \C$ be Waldhausen subcategories. Then the exact functor
$$E(\A, \C, \B) \to \A \times \B, \ (A \rightarrowtail C \twoheadrightarrow B) \mapsto (A, B),$$ 
induces a homotopy equivalence $K(E(\A, \C, \B)) \xrightarrow{\simeq} K(\A) \times K(\B)$. As a consequence, the exact 
functors 
$$E(\A, \C, \B) \to \C, \ (A \rightarrowtail C \twoheadrightarrow B) \mapsto C,$$
and
$$E(\A, \C, \B) \to \C, \ (A \rightarrowtail C \twoheadrightarrow B) \mapsto A \vee B,$$
induce homotopic maps in $K$-theory. 
\end{theorem}

\begin{remark} 
Theorem \ref{addit-thm-2} holds more generally in the case where we simply have exact inclusion functors 
$\A \hookrightarrow \C$ and $\B \hookrightarrow \C$.   
\end{remark}

\subsection{The Approximation Theorem} The Approximation Theorem provides a useful method for showing that an exact functor $F: \C \to \C'$ induces a homotopy  equivalence in $K$-theory. The original formulation 
of Waldhausen \cite[Theorem 1.6.7]{Wa} stated two approximation properties for $F$ as criteria for a $K$-equivalence. Starting with the seminal work of Thomason--Trobaugh \cite{TT}, later treatments of the theorem studied these 
properties from the viewpoint of homotopical algebra, and connected the theorem with the invariance of $K$-theory under derived equivalences 
or equivalences of homotopy theories (see \cite{BM2} and \cite{Ci}). 

\smallskip

The following version of the Approximation Theorem for good Waldhausen categories
is due to Cisinski \cite[Proposition 2.14]{Ci}.

\begin{theorem}[Approximation Theorem] \label{approx-thm}
Let $F: \C \to \C'$ be an exact functor between good Waldausen categories. Suppose that 
$F$ satisfies the following two approximation properties:
\begin{itemize}
\item[(App1)] A morphism $f$ in $\C$ is a weak equivalence if and only if $F(f)$ is a weak equivalence in $\C'$.

\item[(App2)] For every morphism $f: F(X) \to Y$ in $\C'$, there is a weak equivalence $j: Y \to Y'$ in $\C'$, 
a morphism $f': X \to X'$ in $\C$ and a weak equivalence $q: F(X') \to Y'$ such that the square in $\C'$
\[
\xymatrix{
F(X) \ar[r]^f \ar[d]_{F(f')} & Y \ar[d]_{\sim}^j \\
F(X') \ar[r]_{\sim}^q & Y'
}
\]
commutes. 
\end{itemize}
Then the induced map $K(F): K(\C) \stackrel{\simeq}{\to} K(\C')$ is a homotopy equivalence. 
\end{theorem}

\begin{remark} \label{remark-approx}
Blumberg--Mandell \cite{BM2} and Cisinski \cite{Ci} have shown that the approximation properties are closely connected with the property that the functor $F \colon \C \to \C'$ induces an equivalence between the underlying homotopy theories. More precisely, if $F$ satisfies the approximation properties (App1) and (App2), then $F$ induces a Dwyer--Kan equivalence between the associated simplicial localizations. Moreover, the converse also holds if we restrict further to \emph{strongly saturated} Waldhausen categories, that is, Waldhausen categories $\C$ which have the property that a morphism $f$ in $\C$ defines an isomorphism in the homotopy category of $\C$ if and only if $f$ is a weak equivalence. See \cite[Th\'eor\`eme 3.25, Proposition 2.8, Th\'eor\`eme 2.9]{Ci} and \cite[Theorem 1.4, Theorem 1.5; also Theorem 6.4]{BM2}. 
\end{remark}

\subsection{The Fibration Theorem} The Fibration Theorem provides a way of obtaining long exact sequences of $K$-groups similarly to Quillen's Localization Theorem for the $K$-theory of abelian categories \cite{Qu}. It relates the $K$-theories of two different Waldhausen category structures on the same underlying category with cofibrations. 

\smallskip

The following version is a slight improvement of Waldhausen's original formulation \cite[Theorem 1.6.4]{Wa}. 

\begin{theorem}[Fibration Theorem] \label{fibration-thm}
Let 
$$\C_v = (\C, \mathrm{co}\C, v \C) \ \ \textit{and} \ \ \C_w = (\C, \mathrm{co} \C, w \C)$$ 
be Waldhausen categories that have the same underlying categories, the same cofibrations
and $v\C \subset w \C$. Suppose that $(\C, \mathrm{co} \C, w \C)$ is good. Let $\C^w$ be the full 
subcategory of $\C$ spanned by the objects $X \in \C$ such that $\ast \to X$ is in $w \C$. 

Then $\C^w= (\C^w, \mathrm{co}\C \cap \C^w, v\C \cap \C^w)$ is a full Waldhausen subcategory of 
$(\C, \mathrm{co}\C, v\C)$ and the sequence of exact inclusion functors 
$$\C^w \to \C_v \to \C_w$$ 
induces a homotopy fiber sequence 
$$K(\C^w) \to K(\C_v) \to K(\C_w).$$
(The null-homotopy of the composition is the canonical one given by the natural weak $w$-equivalence from the terminal object.)
\end{theorem}
\begin{proof}
This is essentially \cite[Theorem 1.6.4]{Wa}. The replacement of Waldhausen's cylinder axiom with the existence of (non-functorial) factorizations
was worked out in \cite[Theorem A.3]{Sc}. The assumption  in \cite[Theorem 1.6.4]{Wa} that $\C_{w}$ satisfies the extension axiom may be omitted 
by making a small modification in the proof that involves an additional application of the Additivity Theorem (see Addendum \ref{addendum_fibration} below).  
\end{proof}

\begin{addendum} \label{addendum_fibration} We follow the notational conventions used in \cite[Theorem 1.6.4]{Wa}. We recall that the proof of \cite[Theorem 1.6.4]{Wa} uses the extension axiom in order to identify $v \overline{w} \s_{\bullet} \C$ with $v \s_{\bullet} F_{\bullet} (\C, \C^w)$. But the inclusion $v \overline{w} \s_{\bullet} \C \subset v \s_{\bullet} F_{\bullet} (\C, \C^w)$ is always a weak equivalence because we have weak equivalences for each $n \geq 0$,
$$
\xymatrix{
& v \s_{\bullet} \C \times v \s_{\bullet} \s_n \C^w  \ar[rd]^{\simeq} \ar[ld]_{\simeq} & \\
v \s_{\bullet} \overline{w}_n \C  = v \overline{w}_n \s_{\bullet} \C  \ar[rr] && v \s_{\bullet} F_n(\C, \C^w)
}
$$
by the Additivity Theorem (cf. the proof of \cite[Proposition 1.5.5]{Wa}). Here $\overline{w}_n \C$ denotes the Waldhausen subcategory of diagrams consisting of $w$-equivalences in $\C$ of the form:
$$c_0 \stackrel{\sim}{\rightarrowtail} c_1 \stackrel{\sim}{\rightarrowtail} \cdots \stackrel{\sim}{\rightarrowtail} c_n$$
with the usual cofibrations and where the ($v$-)weak equivalences are defined pointwise. $\overline{w}_n \C$ has a Waldhausen subcategory which is  identified with $\C$, embedded as constant diagrams, and a Waldhausen subcategory identified with $\s_{n} \C^w$, embedded as diagrams with $c_0 = *$. The obvious projection functor $E(\C, \overline{w}_n \C, \s_n \C^w) \to \overline{w}_n \C$ admits a section $\overline{w}_n \C \to E(\C, \overline{w}_n \C, \s_n \C^w)$. Then the pair of exact functors 
$$\vee \colon \C \times \s_n \C^w \rightarrow E(\C, \overline{w}_n \C, \s_n \C^w) \to \overline{w}_n \C$$
and 
$$\overline{w}_n \C \to E(\C, \overline{w}_n \C, \s_n \C^w) \to \C \times \s_n \C^w$$
induces inverse homotopy equivalences in $K$-theory as a consequence of the Additivity Theorem (Theorem \ref{addit-thm-2}).
\end{addendum}

\section{The category of $\A$-filtered objects in $\C$}  \label{filtered-objects}

\subsection{The Waldhausen category $\s_n \C_{\A}$} Let $\C$ be a Waldhausen category. Consider the functor $[n+1] \to [n]$ which sends $i \mapsto i$, if $i \leq n$, and $n+1 \mapsto n$. This defines a functor 
$$\mathrm{Ar}[n+1] \to \mathrm{Ar}[n]$$
which induces an exact inclusion functor of Waldhausen categories 
$$i_n: \s_n \C \hookrightarrow \s_{n+1} \C.$$
The functor $i_n$ simply repeats the last column of a staircase diagram -- it is the $n$-th degeneracy map. We define 
$$\s_{\infty} \C \colon = \mathrm{colim} \{ \C \stackrel{i_1}{\to} \s_2 \C \stackrel{i_2}{\to} \cdots \stackrel{i_n}{\to} \s_{n+1} \C \stackrel{i_{n+1}}{\longrightarrow} \cdots \}.$$
The subcategory of cofibrations (resp. weak equivalences) in $\s_{\infty} \C$ is defined to be the colimit of the 
categories of cofibrations (resp. weak equivalences) in $\s_n \C$ for all $n \geq 1$. The following proposition is 
now an easy observation. 

\begin{proposition} \label{good-1}
Let $\C$ be a Waldhausen category. Then $\s_{\infty} \C$ endowed with the subcategories of cofibrations and weak 
equivalences from the Waldhausen categories $\s_n \C$, for $n \geq 1$, is a Waldhausen category. Moreover, 
if $\C$ is good, then so is $\s_{\infty} \C$.
\end{proposition}
 
Let $\A = (\A_i)_{i \geq 1}$ be an (ordered) collection of Waldhausen subcategories of $\C$. We emphasize here that repetitions among the $\A_i$'s are allowed. In particular, the case where $\A_i = \A$ for a single Waldhausen subcategory $\A \subset \C$ will be an important example. 

We consider a subcategory $\s_n \C_{\A}$ of $\s_n \C$ defined as follows. The objects of $\s_n \C_{\A}$ are those objects of $\s_n \C$
$$X_{\bullet, \bullet}: \mathrm{Ar}[n] \to \C$$ 
such that $X_{i-1, i} \in \A_i$ for $1 \leq i  \leq n$. The morphisms in $\s_n \C_{\A}$ are morphisms $F_{\bullet, \bullet}: X_{\bullet, \bullet} \to X'_{\bullet, \bullet}$ 
in $\s_n \C$ such that $F_{i-1,i}$ is in $\A_i$ for each $i \geq 1$. This defines a Waldhausen subcategory of $\s_n \C$ for any $n \geq 1$.  Note that $\s_1 \C_{\A}$ is canonically identified with $\A_1$.

The exact inclusion $i_n$ restricts to an exact inclusion 
between the respective Waldhausen subcategories,
$$i_{n, \A}: \s_n \C_{\A} \hookrightarrow \s_{n+1} \C_{\A}.$$
Similarly, we consider the corresponding Waldhausen subcategory of $\s_{\infty} \C$:
$$\s_{\infty} \C_{\A} \colon = \mathrm{colim} \{ \A_1 \xrightarrow{i_{1, \A}} \s_2 \C_{\A} \xrightarrow{i_{2, \A}} 
\cdots \xrightarrow{i_{n, \A}} \s_{n+1} \C_{\A} \longrightarrow \cdots \}.$$

Given a collection of Waldhausen subcategories $\A = (\A_i)_{i \geq 1}$ of $\C$, we write $\A[n]$ for the new collection of Waldhausen categories obtained after 
shifting by $n$, i.e., $\A[n]_i = \A_{n + i}$ for $i \geq 1$.

\begin{proposition} \label{ho-eq-1}
Let $\C$ be a Waldhausen category and let $\A = (\A_i)_{i \geq 1}$ be a collection of Waldhausen subcategories of $\C$. For each $n \geq 1$, the exact functor 
$$q_n: \s_n \C_{\A} \longrightarrow \prod_{1}^{n} \A_i \ , \ (X_{\bullet, \bullet}) \mapsto (X_{0,1}, \cdots, X_{n-1, n}),$$
induces a homotopy equivalence
$$K(\s_n \C_{\A}) \xrightarrow{\simeq} \prod_{1}^n K(\A_i).$$
Moreover, these induce also a homotopy equivalence $K(\s_{\infty} \C_{\A}) \simeq \underrightarrow{\mathrm{colim}}_n \ \prod_{1}^n K(\A_i).$
\end{proposition}
\begin{proof} 
The claim is obvious for $n = 1$. For $n \geq 2$, we consider the following inclusions of Waldhausen subcategories: $\A_1 \subset \s_n \C_{\A}$, embedded as constant filtered objects, and $\s_{n-1} \C_{\A[1]} \subset \s_n \C_{\A}$, embedded as filtered objects that begin with the zero object in $\A_1$. In addition, these inclusion functors admit retractions, given respectively by the exact functors:
$$r_1 \colon \s_n \C_{\A} \to \A_1 \subset \s_n \C_{\A}, \ \ \ X_{\bullet, \bullet} \mapsto (\ast \rightarrowtail X_{0,1} \stackrel{=}{\rightarrowtail} \cdots \stackrel{=}{\rightarrowtail} X_{0,1}),$$
$$r_2 \colon \s_n \C_{\A} \to \s_{n-1} \C_{\A[1]} \subset \s_n \C_{\A}, \ \ X_{\bullet, \bullet} \mapsto (\ast \rightarrowtail \ast \rightarrowtail X_{1,2} \rightarrowtail \cdots \rightarrowtail X_{1,n}),$$
which assemble to define an exact functor
$$\s_n \C_{\A} \longrightarrow E(\A_1, \s_n \C_{\A}, \s_{n-1} \C_{\A[1]}),\ \ X_{\bullet, \bullet} \mapsto \big(r_1(X_{\bullet, \bullet}) \rightarrowtail X_{\bullet, \bullet} \twoheadrightarrow r_2(X_{\bullet, \bullet})\big).$$
Then the Additivity Theorem (Theorem \ref{addit-thm-2}) shows that the pair of exact functors 
$$(r_1, r_2) \colon \s_n \C_{\A} \rightleftarrows \A_1 \times \s_{n-1} \C_{\A[1]} \colon \vee$$
induces inverse homotopy equivalences in $K$-theory. Thus, we obtain inductively homotopy equivalences 
$$K(\s_n \C_{\A}) \simeq K(\A_1) \times K(\s_{n-1} \C_{\A[1]}) \simeq \cdots \simeq \prod_1^n K(\A_i).$$
By direct inspection, we see that the composite homotopy equivalence is defined by the exact functor
\begin{equation} \label{identification-1}
(X_{\bullet, \bullet}) \mapsto (X_{0,1}, X_{1,2}, \cdots, X_{n-1, n}).
\end{equation}
A homotopy inverse $\prod_1^n K(\A_i) \to K(\s_n \C_{\A})$ is given by the exact functor that includes those filtered 
objects which are defined by successive trivial cofiber sequences: 
$$(A_1, A_2, \cdots, A_n) \mapsto (\ast \rightarrowtail A_1 \rightarrowtail A_1 \vee A_2 \rightarrowtail \cdots \rightarrowtail \bigvee_{1 \leq i \leq n} A_i).$$
Moreover, we have commutative diagrams of exact functors:
$$
\xymatrix{
\s_n \C_{\A} \ar@/^/[d] \ar[r]^{i_{n, \A}} & \s_{n+1} \C_{\A} \ar@/^/[d] \\
\prod_1^{n-1} \A \ar@/^/[u] \ar[r] & \prod_1^n \A \ar@/^/[u] 
}
$$ 
where the bottom functor is the canonical inclusion functor, given on objects by $(A_1, \cdots, A_{n-1}) \mapsto (A_1, \cdots, A_{n-1}, \ast)$. Then the
case $n = \infty$ follows immediately since $\underrightarrow{\mathrm{hocolim}}_n \ K(\s_n \C_{\A}) \xrightarrow{\simeq} K(\s_{\infty} \C_{\A}).$
\end{proof}

\subsection{The Waldhausen category $\sgl_n \C_{\A}$} 
Let $\C$ be a Waldhausen category and let $\A = (\A_i)_{i \geq 1}$ be a collection of Waldhausen subcategories of $\C$. For any $n \geq 1$ or $n=\infty$, the category $\s_n \C_{\A}$  can be endowed with the following weaker type of weak equivalence. 

\begin{definition} \label{global-weq-def} Let $n \geq 1$ be an integer or $n=\infty$. A morphism $F_{\bullet, \bullet}: (X_{\bullet, \bullet}) \to (Y_{\bullet, \bullet})$ in $\s_n \C_{\A}$ (resp. in $\s_{\infty} \C_{\A}$) is called a \emph{total 
weak equivalence} if $F_{0, n}: X_{0,n} \to Y_{0,n}$ is a weak equivalence in $\C$ (resp. for all large enough $n$). 
\end{definition}

Every weak equivalence in $\s_n \C_{\A}$, for $n \in \{1, 2, \cdots, \infty\}$, is also a total weak 
equivalence; this new class of weak equivalences defines a new Waldhausen category structure on the underlying 
category of $\s_{n} \C_{\A}$ (with the same subcategory of cofibrations).
We denote this new Waldhausen category by $\sgl_n \C_{\A}$. Note that the inclusion functors
$$\widehat{i}_{n, \A}: \sgl_n \C_{\A} \hookrightarrow \sgl_{n+1} \C_{\A}$$
are again exact and their colimit is $\sgl_{\infty} \C_{\A}$. The Waldhausen category $\sgl_{\infty} \C_{\A}$ is 
exactly the Waldhausen category of filtrations (or ``unscrewings'')  of objects in $\C$ by objects in the (ordered) collection of subcategories $\A$. 

\medskip

In order to ensure that the Waldhausen categories $\s_n \C_{\A}$ and $\sgl_n \C_{\A}$ are good when $\C$ is, it will be convenient 
to assume in addition that $\C$ admits a cylinder functor which satisfies the cylinder axiom in the sense of \cite{Wa}. This structure equips $\C$ with factorizations 
which are functorial in $\C^{[1]}$ and additionally satisfy certain exactness properties. In particular, restricting these functorial factorizations to the morphisms of 
the form $(X \to \ast)$ yields an exact functor (``cone''): 
$$C \colon \C \to \C, \ \ X \mapsto C(X),$$
and also an exact functor (``suspension''):
$$\Sigma \colon \C \to \C, \ \ X \mapsto \Sigma(X) \colon = C(X)/X.$$ 
We refer to \cite[1.6]{Wa} for more details. The assumption about the existence of cylinder functors is not necessary in all of our proofs, but it will be convenient for some main examples of Waldhausen pairs $(\C, \A)$ for which the inclusion $\A \subset \C$ is 
not homotopically fully faithful.     

\begin{definition} \label{good-Wald-pair-def}
A Waldhausen pair $(\C, \A)$ is called \emph{good} if $\C$ has  the 2-out-of-3 property and admits a cylinder functor which satisfies the cylinder axiom 
and restricts to a cylinder functor on $\A$. 
\end{definition}

Note that if $(\C, \A)$ is a good Waldhausen pair, then both $\C$ and $\A$ are good Waldhausen categories. We will be also interested in the following type of a Waldhausen pair, which may fail to be good in general, but it has different strong homotopical 
properties.    

\begin{definition} \label{replete-def}
A Waldhausen subcategory $\A \subset \C$ is called \emph{replete} if it is a full subcategory and has the following property: given $X \in \A$, then 
any object $Y \in \C$ which is weakly equivalent to $X$ is also in $\A$. In this case, we also say that the Waldhausen pair $(\C, \A)$ is replete. 
\end{definition} 

\begin{lemma} \label{additional-props}
Let $\C$ be a good Waldhausen category and let $\A = (\A_i)_{i \geq 1}$ be a collection of Waldhausen subcategories of $\C$. Suppose that $\C$ has 
a cylinder functor which satisfies the cylinder axiom. 
\begin{itemize}
 \item[(i)] The Waldhausen categories $\s_n \C_{\A}$ and $\sgl_n \C_{\A}$ have the 2-out-of-3 property for any $n \in \{1,2,\dots, \infty\}$. 
\item[(ii)] Suppose that $(\C, \A_i)$ is replete for every $i \geq 1$ and that $\Sigma \colon \C \to \C$ sends $\A_i$ to $\A_{i+1}$. Then $\sgl_{\infty} \C_{\A}$ admits factorizations. In particular, $\sgl_{\infty} \C_{\A}$ is a good Waldhausen category.  
 \item[(iii)] Let $n \geq 1$ be an integer or $n=\infty$. Suppose that $(\C, \A_i)$ is a good Waldhausen pair for $1 \leq i \leq n$. Then $\s_n \C_{\A}$ and $\sgl_n \C_{\A}$ admit (functorial) factorizations. In particular, the Waldhausen categories $\s_n \C_{\A}$ and $\sgl_n \C_{\A}$ are good Waldhausen categories. 
\end{itemize}
\end{lemma}
\begin{proof}
(i) is obvious. (ii): Let $F_{\bullet, \bullet} \colon X_{\bullet, \bullet} \to Y_{\bullet, \bullet}$ be 
a morphism in $\sgl_{\infty} \C_{\A}$. By Proposition \ref{good-1}, there is a factorization of the morphism in $\s_{\infty} \C$: 
$$F_{\bullet, \bullet} + \mathrm{Id} \colon X_{\bullet, \bullet} \vee Y_{\bullet, \bullet} \rightarrow Y_{\bullet, \bullet}$$
into a cofibration followed by a weak equivalence (both in the sense of $\s_{\infty} \C$)
$$X_{\bullet, \bullet} \vee Y_{\bullet, \bullet} \rightarrowtail Z_{\bullet, \bullet} \stackrel{\sim}{\rightarrow} Y_{\bullet, \bullet}.$$
Then we define a new object $Z'_{\bullet, \bullet} \in \s_{\infty} \C$ by 
$$Z'_{0, i} : = X_{0,i} \cup_{X_{0,i-1}} Z_{0,i-1} \cup_{Y_{0,i-1}} Y_{0,i}.$$
Using our assumptions on $(\C, \A)$, it follows that the object $Z'_{\bullet, \bullet}$ is in $\s_{\infty} \C_{\A}$. Moreover, the canonical morphisms $(X_{0,i} \to Z'_{0,i} \to Y_{0,i})$ define the required factorization in $\sgl_{\infty}\C_{\A}$:
$$X_{\bullet, \bullet} \rightarrowtail Z'_{\bullet, \bullet} \xrightarrow{\sim} Y_{\bullet, \bullet}.$$
  (iii): The functorial factorizations in $\C$, given by the cylinder functor, produce functorial factorizations in $\s_n \C$ \cite[1.6.1]{Wa}. 
Using our assumptions on $(\C, \A_i)$, it follows that these restrict to functorial factorizations on the Waldhausen subcategory $\s_n \C_{\A}$, as required. The factorizations in $\s_n \C_{\A}$ define also factorizations in $\sgl_n\C_{\A}$.
\end{proof}

\subsection{A homotopy fiber sequence} 
For any $n \in \{1,2, \cdots, \infty\}$, let $\s^w_n \C_{\A}$ denote the full Waldhausen subcategory of $\s_n \C_{\A}$ spanned by those objects $X_{\bullet, \bullet} \in \s_n \C_{\A}$ which are weakly equivalent in $\sgl_n \C_{\A}$ to the zero object,  that is, the morphism $\ast \to X_{\bullet, \bullet}$ is a 
total weak equivalence. Clearly, if $\s_n \C_{\A}$ is good, then so is $\s^w_n \C_{\A}$. The exact inclusion functors $i_n$ restrict to exact functors 
$$i^w_{n, \A} \colon \s^w_n \C_{\A} \hookrightarrow \s^w_{n+1} \C_{\A}$$
whose colimit is $\s^w_{\infty} \C_{\A}$. For $n=1$, $\s^w_1 \C_{\A}$ is identified with the Waldhausen subcategory of $\A_1$ which consists of the weakly trivial objects. 

\begin{proposition} \label{fiber seq} 
Let $\C$ be a good Waldhausen category, $\A = (\A_i)_{i \geq 1}$ a collection of Waldhausen subcategories of $\C$, and let $n \in \{1, 2, \cdots, \infty\}$. 
Suppose that $\sgl_n \C_{\A}$ is a good Waldhausen category (see, e.g.,  Lemma \ref{additional-props}). Then the exact functors 
$\s^w_n \C_{\A} \hookrightarrow \s_n \C_{\A} \to \sgl_n \C_{\A}$
induce a homotopy fiber sequence 
$$K(\s^w_n \C_{\A}) \to K(\s_n \C_{\A}) \to K(\sgl_n \C_{\A}).$$
\end{proposition}
\begin{proof}
This is a direct application of the Fibration Theorem (Theorem \ref{fibration-thm}). 
\end{proof}

\section{Admissible Waldhausen Pairs} \label{admissible-Wald-pairs}

\subsection{Basic definitions and properties} \label{basics-admissibility}
Let $(\C, \A)$ be a good Waldhausen pair. For the purpose of the devisage theorem in the next section, we will need to consider an additional \emph{ad hoc} condition on $(\C, \A)$ which states an identification of the $K$-theory of the Waldhausen category $\s^w_{\infty} \C_{\A}$. For each $n \geq 1$, we have an exact functor 
\begin{equation*}
q^w_n \colon \s^w_n \C_{\A} \longrightarrow \prod_{1}^{n-1} \A, \ \ (X_{\bullet, \bullet}) \mapsto (X_{0,1}, \cdots, X_{n-2, n-1}). 
\end{equation*}
Each functor $q^w_n$ admits a section given by the exact functor
$$j_n \colon \prod_{1}^{n-1} \A \to \s^w_n \C_{\A}, \  \ (A_1, \cdots, A_{n-1}) \mapsto (\ast \rightarrowtail A_1 \rightarrowtail \cdots \rightarrowtail \bigvee_1^{n-1} A_{i} \rightarrowtail \bigvee_1^{n-1} C(A_{i})).$$
As a consequence, the induced map $K(j_n)$ defines a section of $K(q^w_n)$. It will be convenient to consider also an exact functor $\tau_n \colon \prod_1^{n-1} \A \to \s^w_n \C_{\A}$ which induces a section only up to a homotopy equivalence. This functor is defined on objects by 
$$\tau_n \colon (A_1, \cdots, A_{n-1}) \mapsto (\ast \rightarrowtail A_1 \rightarrowtail A_2 \vee CA_1 \rightarrowtail \cdots \rightarrowtail A_{n-1} \vee \bigvee_1^{n-2} CA_{i} \rightarrowtail \bigvee_1^{n-1} CA_i),$$
and similarly on morphisms. The composite functor $q^w_n \circ \tau_n$ is given on objects by 
$$(A_1, \cdots, A_{n-1}) \mapsto (A_1, A_2 \vee \Sigma A_1, \cdots, A_{n-1} \vee \Sigma A_{n-2})$$
which induces a homotopy equivalence in $K$-theory. More specifically, the map in $K$-theory is identified with the homotopy equivalence
$$ (\pi_1, \pi_2 - \pi_1, \cdots, \pi_{n-1} - \pi_{n-2}): \prod_1^{n-1} K(\A) \to \prod_1^{n-1} K(\A)$$
where $\pi_i$ denotes the projection onto the $i$-th factor and the sum/difference is induced by the loop sum. Clearly each of the maps $K(q^w_n)$, $K(j_n)$, and $K(\tau_n)$, is a homotopy equivalence if any one of them is a homotopy equivalence. 

\medskip

The functors $\{\tau_n\}_{n \geq 1}$ are compatible with respect to $n \geq 1$, in the sense that the following diagrams of exact functors commute
$$
\xymatrix{
\s^w_n \C_{\A} \ar[r]^{i^w_{n, \A}} & \s^w_{n+1} \C_{\A}. \\
\prod_1^{n-1} \A \ar[u]_{\tau_n} \ar[r] & \prod_1^n \A \ar[u]_{\tau_{n+1}}
}
$$ 
Here the bottom functor is the canonical inclusion functor, given on objects by $(A_1, \cdots, A_{n-1}) \mapsto (A_1, \cdots, A_{n-1}, \ast)$. As a consequence, we obtain an exact 
functor 
$$\tau_{\infty} \colon \underrightarrow{\colim}_n \prod_1^{n-1} \A \longrightarrow \s^w_{\infty} \C_{\A}.$$

\begin{definition} \label{admissible-def}
Let $(\C, \A)$ be a good Waldhausen pair. We say that $(\C, \A)$ is \emph{admissible} if the exact functor $\tau_{\infty}$ induces a homotopy equivalence:
$$K(\tau_{\infty}) \colon \underrightarrow{\colim}_n \prod_1^{n-1} K(\A) \stackrel{\simeq}{\longrightarrow} K(\s^w_{\infty} \C_{\A}).$$
\end{definition} 

\begin{remark}
The comments above imply that $K(\tau_n)$ is (split) injective on homotopy groups. As a consequence, $K(\tau_{\infty})$ is also always injective on homotopy groups. 
\end{remark} 

\begin{remark}
The functors $\{j_n\}_{n \geq 1}$ are not compatible with respect to the natural inclusion functors $i^w_{n ,\A}$.  We consider the exact functor
$$\xi_n \colon \prod_1^{n-1} \A \to \prod_1^{n} \A, \ \ (A_1, \cdots, A_{n-1}) \mapsto (A_1, \cdots, A_{n-1}, \bigvee_1^{n-1}  \Sigma (A_i)).$$ 
The map in $K$-theory induced by $\xi_n$ can be identified with the map 
$$(\pi_1, \cdots, \pi_{n-1}, - \sum_1^{n-1} \pi_i ): \prod_1^{n-1} K(\A) \to \prod_1^{n} K(\A)$$
where $\pi_i$ denotes the projection onto the $i$-th factor. Using the Additivity Theorem (Theorem \ref{addit-thm-2}), it can be shown that the induced diagram in $K$-theory
\begin{equation} \label{j_n-maps}
\xymatrix{
K(\s^w_n \C_{\A} )\ar[rr]^{K(i^w_{n, \A})} && K(\s^w_{n+1} \C_{\A}) \\
\prod_1^{n-1} K(\A) \ar[rr]_{K(\xi_n)} \ar[u]^{K(j_n)} && \prod_1^n K(\A) \ar[u]_{K(j_{n+1})} 
}
\end{equation}
commutes up to a preferred homotopy. Thus, we may obtain also in this way a map as $n \to \infty$: 
$$J_{\infty} \colon \underrightarrow{\colim}_{(\xi_n)} \prod_1^{n-1} K(\A) \longrightarrow K(\s^w_{\infty}\C_{\A}).$$ 
We do not know if the corresponding diagram to \eqref{j_n-maps} where the vertical maps are replaced by the retraction maps $K(q^w_n)$ is also homotopy commutative -- this is indeed related to the admissibility assumption. 
\end{remark}

\medskip 

We do not know an example of a good Waldhausen pair $(\C, \A)$ which is not admissible. It is likely that the examples studied by Antieau--Barthel--Gepner \cite{ABG} in connection with Rognes' question are such examples (see \cite[Section 6]{ABG}). In the following sections, we will show several classes of examples where the admissibility assumption is satisfied, using arguments specific to each case. It would be interesting to identify further sufficient homotopical conditions on $(\C, \A)$ which are easy to verify and ensure that $(\C, \A)$ is admissible. 

\subsection{Criteria for admissiblity} \label{criteria-subsec} 
Let $(\C, \A)$ be a good Waldhausen pair. We let $\s^w_n \C_{\A}^1$ denote the Waldhausen category whose category with cofibrations is the same as that of $\s^w_n \C_{\A}$ and a morphism $F_{\bullet, \bullet} \colon X_{\bullet, \bullet} \to Y_{\bullet, \bullet}$ is a weak equivalence if $F_{0,1}$ is a weak equivalence in $\A$. This is again a good Waldhausen category. 

For $n \geq 2$, there are exact functors as follows:
$$p_{n}^1 \colon \s^w_n \C_{\A}^1 \to \A, \ \ (X_{\bullet, \bullet}) \mapsto X_{0,1},$$
$$s \colon \A \to \s^w_n \C_{\A}^1, \ \ A \mapsto (\ast \rightarrowtail A = A = \cdots = A \rightarrowtail C(A)).$$
Note that the composite functor $p_{n}^1 \circ s$ is the identity functor.

Let $\mathrm{iso} \ \s^w_n \C_{\A}^1$ (resp. $\mathrm{iso} \ \A$) denote the associated Waldhausen category whose underlying category with cofibrations is that of $\s^w_n \C_{\A}^1$ (resp. $\A$), and a morphism $F_{\bullet, \bullet} \colon X_{\bullet, \bullet} \to Y_{\bullet, \bullet}$ is a weak equivalence if $F_{0,1}$ is an isomorphism (resp. the weak equivalences in $\mathrm{iso} \ \A$ are 
the isomorphisms).  The functors defined above are exact also with respect to these Waldhausen category structures. 

\begin{proposition} \label{criterion}
Let $(\C, \A)$ be a good Waldhausen pair. Then $(\C, \A)$ is admissible if any of the following conditions holds.
\begin{itemize}
\item[(1)] The map $K(p_{n}^1) \colon K(\s^w_n \C_{\A}^1) \to K(\A), \ \ (X_{\bullet, \bullet}) \mapsto X_{0,1}$, is a homotopy equivalence for every $n \geq 2$.  
\item[(2)]  The map $$K(p_{n}^1) \colon K(\mathrm{iso} \ \s^w_n \C_{\A}^1) \to K(\mathrm{\iso}\ \A), \ \ (X_{\bullet, \bullet}) \mapsto X_{0,1},$$
is a homotopy equivalence for every $n \geq 2$. 
\end{itemize}
\end{proposition}
\begin{proof}
(1) We proceed by induction on $n \geq 2$ and show that $K(q^w_k)$ is a homotopy equivalence for $k \leq n$ if and only if $K(p_{k}^1)$ is a homotopy equivalence for $k \leq n$. The claim is obvious for $n = 2$. For the inductive step, consider the full Waldhausen subcategory $\E_n \subset \s^w_n \C_{\A}$
which consists of objects $(X_{\bullet, \bullet})$ such that $X_{0,1}$ is weakly trivial. Then the Fibration Theorem (Theorem \ref{fibration-thm}) implies that there is a homotopy fiber sequence 
$$K(\E_n) \to  K(\s^w_n \C_{\A}) \to K(\s^w_n \C_{\A}^1).$$
There are exact functors 
$$\rho \colon \E_n \to \s^w_{n-1} \C_{\A}, \ \ (X_{\bullet, \bullet}) \mapsto (\ast \rightarrowtail X_{1,2} \rightarrowtail \cdots \rightarrowtail X_{1,n})$$
$$\iota \colon \s^w_{n-1} \C_{\A} \to \E_n, \ \ (Y_{\bullet, \bullet}) \mapsto (\ast = \ast \rightarrowtail Y_{0,1} \rightarrowtail \cdots \rightarrowtail Y_{0, n-1}).$$
The composition $\rho \circ \iota$ is identified with the identity functor. The composition $\iota \circ \rho$ is weakly equivalent to the identity functor via the following natural weak equivalence of filtered objects:
\[
 \xymatrix{
 X_{0,1} \ar@{>->}[r] \ar[d]^{\sim} & X_{0,2} \ar@{>->}[r] \ar[d]^{\sim} & \cdots \ar@{>->}[r] & X_{0,n} \ar[d]^{\sim} \\
 \ast \ar@{>->}[r] & X_{1,2} \ar@{>->}[r] & \cdots \ar@{>->}[r] & X_{1,n}.
 }
\]
Therefore, $K(\E_n) \simeq K(\s^w_{n-1} \C_{\A}) \simeq \prod_{1}^{n-2} K(\A)$, using the inductive assumption. Now we consider the following homotopy commutative diagram: 
\[
 \xymatrix{
K(\E_n) \ar[r] &  K(\s^w_n \C_{\A}) \ar[r] & K(\s^w_n\C_{\A}^1) \\
\prod_{1}^{n-2} K(\A) \ar[r] \ar[u]^{\simeq}_{K(\iota j_{n-1})} & \prod_1^{n-1} K(\A) \ar[r]^(.6){\pi_1} \ar[u]_{K(j_n)} & K(\A) \ar[u]_{K(s)}
}
\]
where the rows are homotopy fiber sequences (of infinite loop spaces). The right square is homotopy commutative because the underlying exact functors are naturally weakly equivalent. Hence the right vertical map is a homotopy equivalence if and only if the middle vertical map is a homotopy equivalence. Equivalently, the map $K(p_{n}^1)$ is a homotopy equivalence if and only if $K(j_n)$ -- and therefore also $K(q^w_n)$ and $K(\tau_n)$ -- is a homotopy equivalence. The result then follows by induction. 

\noindent (2) Let $\D_n \subset \mathrm{iso} \ \s^w_n \C_{\A}^1$ be the full Waldhausen subcategory spanned by the objects $X_{\bullet, \bullet}$ such that $X_{0,1}$ is weakly trivial. Consider the following diagram in $K$-theory:
$$
\xymatrix{
K(\D_n) \ar[r] \ar[d] & K(\mathrm{iso} \ \s^w_n \C_{\A}^1) \ar[r] \ar[d] & K(\s^w_n \C_{\A}^1) \ar[d] \\
K(\mathrm{iso} \ \A^w) \ar[r] & K(\mathrm{iso}\ \A) \ar[r] & K(\A)
}
$$
where the rows define homotopy fiber sequences (of infinite loop spaces) by the Fibration Theorem (Theorem \ref{fibration-thm}).  The vertical maps are induced by the functor $p_{n}^1 \colon X_{\bullet, \bullet} \mapsto X_{0,1}$. The map $K(\D_n) \to K(\mathrm{iso} \ \A^w)$ is a homotopy equivalence with homotopy inverse 
induced by the functor $A \mapsto (\ast \rightarrowtail A = \cdots = A)$. Hence the right vertical map is a homotopy equivalence if (and only if) the middle 
vertical map is a homotopy equivalence. The result then follows from (1). 
\end{proof}

\begin{remark}
As a consequence of the Approximation Theorem (Theorem \ref{approx-thm} and Remark \ref{remark-approx}), the criterion of Proposition \ref{criterion}(1) is satisfied if $\A$ is strongly saturated and $s \colon \A \to \s^w_n \C_{\A}^1$ induces an equivalence between the homotopy categories for every $n \geq 2$. This condition expresses a property of the Waldhausen pair $(\C, \A)$ with regard to cofibrations to a weakly trivial object in $\C$ whose cofiber is in $\A$. We emphasize that this property is not immediate, especially, because the inclusion $\A \subset \C$ is not assumed to be homotopically fully faithful in general.  
\end{remark}

\subsection{Example: Replete Waldhausen pairs} \label{replete-subsec}
The next proposition shows that replete Waldhausen pairs (Definition \ref{replete-def}) are examples of admissible Waldhausen pairs.

\begin{proposition} \label{replete-example} 
Let $(\C, \A)$ be a good Waldhausen pair which is also replete. Then $(\C, \A)$ is admissible. 
\end{proposition}
\begin{proof}
We claim that the exact functor $q^w_n \colon \s^w_n \C_{\A} \to \prod_1^{n-1} \A$ induces a homotopy equivalence in $K$-theory for any $n \geq 1$. The claim is obvious for $n=1$, so we 
may assume that $n > 1$. Let $\s_{n-1} (\C_{\A}, \Sigma^{-1} \A)$ denote the full Waldhausen subcategory of $\s_{n-1} \C_{\A}$ spanned by
objects $(X_{\bullet, \bullet})$ such that the suspension of $X_{0, n-1}$ is in $\A$. In detail, this means that there is a factorization 
$$X_{0, n-1} \rightarrowtail Z \stackrel{\sim}{\to} \ast$$
such that the cofiber of the first morphism is in $\A$. (Since $\A \subset \C$ is replete, this property is independent of the choice of the factorization.) The full 
Waldhausen subcategory $\s_{n-1}(\C_{\A}, \Sigma^{-1} \A)$ has the 2-out-of-3 property and admits factorizations because $\s_{n-1} \C_{\A}$ has these properties. There is an exact functor 
$$\mathscr{U}: \s^w_{n} \C_{\A} \to \s_{n-1}(\C_{\A}, \Sigma^{-1} \A)$$
which simply forgets the last column. 
We show that $\mathscr{U}$ induces a homotopy equivalence in $K$-theory by checking that it has properties (App1) and (App2) of the Approximation Theorem 
(Theorem \ref{approx-thm}). $\mathscr{U}$ obviously satisfies 
(App1). To verify (App2), let 
$$f: \mathscr{U}(X_{\bullet, \bullet}) \to Z_{\bullet, \bullet}$$ 
be a morphism in $\s_{n-1}(\C_{\A}, \Sigma^{-1} \A)$ and consider a factorization 
$$Z_{0,n-1} \cup_{X_{0, n-1}} X_{0,n} \rightarrowtail Y \stackrel{\sim}{\to} \ast.$$
Then let $Z'_{\bullet, \bullet}$ be the object of $\s^w_n \C_{\A}$ 
that corresponds to the filtered object 
$$\ast \rightarrowtail Z_{0,1} \rightarrowtail Z_{0,2} \rightarrowtail \cdots \rightarrowtail 
Z_{0, n-1} \rightarrowtail Y$$
by making choices of cofibers. There is an obvious morphism $f': X_{\bullet, \bullet} \to Z'_{\bullet, \bullet}$ and 
a weak equivalence $q: \mathscr{U}(Z'_{\bullet, \bullet}) \xrightarrow{=} Z_{\bullet, \bullet}$ such that $f = q \circ \mathscr{U}(f')$. 
Thus, (App2) holds, and therefore $\mathscr{U}$ induces a homotopy equivalence in $K$-theory:
$$K(\mathscr{U}) \colon K\big(\s^w_{n} \C_{\A}\big) \stackrel{\simeq}{\to} K\big(\s_{n-1}(\C_{\A}, \Sigma^{-1} \A)\big).$$
This proves our original claim for $n=2$, since $\s_1 (\C_{\A}, \Sigma^{-1} \A)$ is identified with $\A$. For general $n$, we may view $\A$ as the full Waldhausen subcategory of constant filtered objects 
in $\s_{n-1} (\C_A,  \Sigma^{-1} \A)$, and $\s_{n-2} (\C_A,  \Sigma^{-1} \A)$ as the full Waldhausen subcategory of those filtered objects with $X_{0,1} = \ast$. 
In addition, these inclusion functors admit retractions, given respectively by the exact functors:
\begin{equation*}
\begin{split}
r_1 \colon \s_{n-1}(\C_{\A}, \Sigma^{-1} \A) \to \A \subset \s_{n-1}(\C_{\A}, \Sigma^{-1} \A) \\
X_{\bullet, \bullet} \mapsto (\ast \rightarrowtail X_{0,1} \stackrel{=}{\rightarrowtail} \cdots \stackrel{=}{\rightarrowtail} X_{0,1})
\end{split} 
\end{equation*}
and
\begin{equation*}
\begin{split}
r_2 \colon \s_{n-1}(\C_{\A}, \Sigma^{-1} \A) \to \s_{n-2}(\C_{\A}, \Sigma^{-1} \A) \subset \s_{n-1}(\C_{\A}, \Sigma^{-1}\A) \\ 
X_{\bullet, \bullet} \mapsto (\ast \rightarrowtail \ast \rightarrowtail X_{1,2} \rightarrowtail \cdots \rightarrowtail X_{1,n-1}),
\end{split}
\end{equation*}
which assemble to define an exact functor
\begin{equation*} 
\begin{split}
\s_{n-1}(\C_{\A}, \Sigma^{-1} \A) \longrightarrow E\big(\A, \s_{n-1}(\C_{\A}, \Sigma^{-1}\A), \s_{n-2}(\C_{\A}, \Sigma^{-1}\A)\big) \\ 
X_{\bullet, \bullet} \mapsto \big(r_1(X_{\bullet, \bullet}) \rightarrowtail X_{\bullet, \bullet} \twoheadrightarrow r_2(X_{\bullet, \bullet})\big).
\end{split}
\end{equation*}
Then the Additivity Theorem (Theorem \ref{addit-thm-2}) shows that the pair of exact functors 
$$(r_1, r_2) \colon \s_{n-1}(\C_{\A}, \Sigma^{-1} \A) \rightleftarrows \A \times \s_{n-2}(\C_{\A}, \Sigma^{-1} \A) \colon \vee$$
induces inverse homotopy equivalences in $K$-theory. Using this homotopy equivalence, we also obtain inductively homotopy equivalences as follows:
$$K\big(\s_{n-1} (\C_A, \Sigma^{-1} \A)\big) \simeq K(\A) \times K\big(\s_{n-2}(\C_{\A}, \Sigma^{-1} \A)\big) \simeq \cdots \simeq \prod_1^{n-1} K(\A).$$
By direct inspection, we observe that the composite homotopy equivalence is given by the functor: 
\begin{equation} \label{identification-2}
(X_{\bullet, \bullet}) \mapsto (X_{0,1}, X_{1,2}, \cdots, X_{n-2, n-1}).
\end{equation}\end{proof}

\begin{remark}
The proof of Proposition \ref{replete-example} works also under some slightly weaker assumptions. It suffices that $\C$ is good, $\A \hookrightarrow \C$ detects cofibrations, and $(\C, \A)$ is replete. In other words, under these assumptions, it is not required to assume that $\C$ admits \emph{functorial} factorizations.
\end{remark}

\subsection{Example: Waldhausen pairs with (WHEP)} \label{whep}

This example is inspired by the assumptions in the d\'evissage theorem of D. Yao \cite{Ya}. Let $(\C, \A)$ be a good Waldhausen pair. We say that $(\C, \A)$ satisfies the \emph{weak homotopy extension property} (WHEP) if the following property is satisfied: given a cofibration $A \rightarrowtail B$ in $\A$ and $X \in \C$ such that $X \xrightarrow{\sim} \ast$, 
then every morphism $C(A) \cup_A B \to X$ extends to a morphism $C(B) \to X$ along the cofibration $C(A) \cup_A B \rightarrowtail C(B)$. 

\begin{proposition}
Let $(\C, \A)$ be a good Waldhausen pair where $\A \subset \C$ is a full subcategory. Suppose that $(\C, \A)$ satisfies the weak homotopy extension property.  Then $(\C, \A)$ is admissible. 
\end{proposition}
\begin{proof}
We show that the condition of Proposition \ref{criterion}(1) is satisfied. To this end, we claim that the exact functor $s \colon \A \to \s^w_n\C_{\A}^1$ 
induces a homotopy equivalence in $K$-theory. This is shown by applying the Approximation Theorem (Theorem \ref{approx-thm}). The functor $s$ 
clearly satisfies (App1). For (App2), we consider a morphism $F_{\bullet, \bullet}  \colon s(A) \to X_{\bullet, \bullet}$ and a factorization in $\A$,
$$A \stackrel{i}{\rightarrowtail} Z \xrightarrow{q, \sim} X_{0,1}.$$
By the weak homotopy extension property, we find an extension of the morphism $C(A) \cup_A Z \to X_{0,n}$ to a morphism $q' \colon C(Z) \to X_{0,n}$. This extension is used to define a morphism $Q_{\bullet, \bullet} \colon s(Z) \to X_{\bullet, \bullet}$ by 
$$
\xymatrix{
Z \ar@{=}[r] \ar[d]_q & Z \ar@{=}[r] \ar[d] & \cdots \ar@{=}[r] & Z \ar[d] \ar@{>->}[r] & C(Z) \ar[d]_{q'} \\
X_{0,1} \ar@{>->}[r] & X_{0,2} \ar@{>->}[r] & \cdots \ar@{>->}[r] & X_{0,n-1} \ar@{>->}[r] & X_{0,n}  
}
$$
and this gives the required factorization of $F_{\bullet, \bullet}$ as the composition 
$$s(A) \stackrel{s(i)}{\rightarrowtail} s(Z) \xrightarrow{Q_{\bullet, \bullet}, \sim} X_{\bullet, \bullet}.$$
Then the claim follows from the Approximation Theorem (Theorem \ref{approx-thm}).
\end{proof}

\begin{example} \label{fibrant-objects-example}
Let $(\C, \A)$ be a good Waldhausen pair where $\A \subset \C$ is a full subcategory. Suppose that there is an exact fully faithful functor $\iota \colon \C \hookrightarrow \mathcal{M}_{c}$ in the full subcategory of cofibrant objects of a pointed model category $\mathcal{M}$. In addition, suppose that for every $X \in \C$ which is weakly trivial, the object $\iota(X) \in \mathcal{M}_c$ is fibrant. Then $(\C, \A)$ has the weak homotopy extension property, and therefore, it is admissible.
\end{example}

\section{Single Type D\'{e}vissage} \label{single-type-dev-sec}

\noindent \textbf{Assumptions.} We fix the following notation and assumptions throughout this section. Let $(\C, \A)$ denote a good Waldhausen pair. We denote also by $\A$ the collection of Waldhausen subcategories $(\A = \A_i)_{i \geq 1}$ which is constant at $\A$. By Lemma \ref{additional-props}, it follows that the Waldhausen categories $\s_n \C_{\A}$ and $\sgl_{n} \C_{\A}$ are good for every $n \in \{1, 2, \dots, \infty\}$.  

\subsection{The d\'{e}vissage condition} \label{devissage-cond-subsec}
There are exact ``evaluation" functors
$$ev_n: \sgl_n \C_{\A} \to \C, \ \ (X_{\bullet, \bullet}) \mapsto X_{0,n}.$$
These are compatible with stabilization along the inclusion functors $\widehat{i}_{n, \A}$, so there is also an induced exact functor, 
$$ev_{\infty}: \sgl_{\infty} \C_{\A} \to \C,$$
which sends a staircase diagram representing a bounded $\A$-filtration of an object $X_{0,n}$, for $n$ large enough, to the object $X_{0,n}$ itself. One of the main 
functions of a d\'{e}vissage condition on $(\C, \A)$ is to ensure that the exact functor $ev_{\infty}$ induces a $K$-equivalence. 

\begin{definition} \label{devissage-cond-def}
We say that $(\C, \A)$ satisfies the \emph{d\'{e}vissage condition} if for every morphism $f \colon X \rightarrow Y$ in $\C$, there is a weak equivalence $g \colon Y \stackrel{\sim}{\rightarrow} Y'$ such that the composition $gf \colon X \to Y'$ 
admits a factorization
  $$X = X_0 \rightarrowtail X_1 \rightarrowtail \cdots \rightarrowtail X_{m-1} \rightarrowtail X_m \stackrel{\sim}{\to} Y'$$
where (some choice of cofiber) $X_i / X_{i-1}$ is in $\A$ for all $i \geq 1$.
\end{definition}

This says that up to weak equivalence every morphism in $\C$, and not just the objects, admit filtrations whose subquotients are  in $\A$. We emphasize that $\A \subset \C$ is not assumed to be replete (or homotopically fully faithful) in general and the d\'evissage condition requires that the subquotients must be strictly in $\A$, i.e., up to isomorphism and not only up to weak equivalence. Even though these filtrations are not assumed to be functorial or invariant in any sense, we will explain in Remark \ref{uniqueness-filtrations} that the d\'evissage condition suffices for the functor $ev_{\infty} \colon \sgl_{\infty} \C_{\A} \to \C$ to be an equivalence of homotopy theories. 

\begin{example} \label{bigger-A}
Suppose that $(\C, \A)$ satisfies the d\'evissage condition. Let $(\C, \A')$ be another 
good Waldhausen pair where $\mathrm{Ob}\A \subseteq \mathrm{Ob} \A'$. Then $(\C, \A')$ also satisfies the d\'evissage condition. 
\end{example}

\begin{example} \label{bigger-C}
Suppose that $(\C, \A)$ satisfies the d\'evissage condition. Let $\C'$ be another Waldhausen category whose underlying category with 
cofibrations is $(\C, co\C)$ and $w\C \subset w\C'$. Let $(\C', \A')$ denote the good Waldhausen pair which corresponds to the subcategory of $\C'$
defined by $\A$. Then $(\C', \A')$ also satisfies the d\'evissage condition. 
\end{example}

\begin{example} \label{A-filtered-objs}
Let $(\C, \A)$ be a good Waldhausen pair and let $n \geq 1$ be an integer or $n=\infty$. Then $(\s_n \C_\A, \s_n \A)$ is again a good Waldhausen pair and satisfies the d\'evissage condition. The required filtrations are constructed using the fact that the objects in $\s_n \C_\A$ have subquotients in $\A$ (cf.  the inductive procedure used in the proof of Proposition \ref{ho-eq-1}).
\end{example}
 
\begin{proposition}\label{devissage-cond-prop} 
Suppose that $(\C, \A)$ satisfies the d\'{e}vissage condition. Then $ev_{\infty} \colon \sgl_{\infty} \C_{\A} \to \C$ 
induces a homotopy equivalence 
$$K(\sgl_{\infty} \C_{\A}) \stackrel{\simeq}{\rightarrow} K(\C).$$
\end{proposition}
\begin{proof}
We claim that the functor $ev_{\infty}: \sgl_{\infty} \C_{\A} \to \C$ satisfies the conditions of the Approximation Theorem (Theorem \ref{approx-thm}). (App1) holds by definition. For (App2), let $(X_{\bullet, \bullet})$ be an object in $\sgl_{\infty} \C_{\A}$ and consider a morphism in $\C$
$$f: ev_{\infty}(X_{\bullet, \bullet}) = X_{0,n} \to Y.$$ 
Using the d\'{e}vissage condition applied to the morphism $f$, there is a weak equivalence $g: Y \stackrel{\sim}{\rightarrow} Y'$ and 
a factorization of the morphism $gf$,
$$X_{0,n} = Z_0 \rightarrowtail Z_1 \rightarrowtail \cdots \rightarrowtail Z_m \stackrel{\sim}{\to} Y'$$
such that (choices of) the cofibers $Z_i/Z_{i-1}$ are in $\A$. Consequently, the combination of the two filtered objects 
$$(X_{\bullet, \bullet'})_{0 \leq \bullet \leq \bullet' \leq n} \ \ \text{and} \ \ (Z_{\bullet})$$ 
can be extended, by making choices of cofibers, to a new object $(Z_{\bullet, \bullet})$ in $\sgl_{\infty} \C_{\A}$. There is an obvious canonical morphism 
in $\sgl_{\infty} \C_{\A}$
$$f': (X_{\bullet, \bullet}) \to (Z_{\bullet, \bullet})$$
whose components $f'_{0, \bullet}$ are given either by identities or by a composition of cofibrations in the factorization of the map $X_{0,n} \to Y'$ considered above. 
Lastly, the weak equivalence $ev_{\infty}(Z_{\bullet, \bullet})= Z_m \stackrel{\sim}{\to} Y'$ fits in a commutative diagram 
\[
 \xymatrix{
 ev_{\infty}(X_{\bullet, \bullet}) \ar[r]^(.6)f \ar[d]_{ev_{\infty}(f')} &  Y \ar[d]^{\sim}_g \\
 ev_{\infty}(Z_{\bullet, \bullet}) \ar[r]^(.6){\sim} &  Y'.
 }
\]
This shows that (App2) is satisfied and then the result follows from Theorem \ref{approx-thm}.
\end{proof}

\begin{remark} \label{uniqueness-filtrations}
As a consequence of the results stated in Remark \ref{remark-approx}, the proof of Proposition \ref{devissage-cond-prop} also shows that the exact functor 
$ev_{\infty} \colon \sgl_{\infty} \C_{\A} \to \C$ induces an equivalence between the underlying homotopy theories. 
\end{remark} 

\subsection{The main theorem} \label{main-thm-1} 
A d\'evissage theorem of single type is a statement asserting that for certain Waldhausen pairs $(\C, \A)$ satisfying a d\'evissage--type condition, the induced map $K(\A) \xrightarrow{\simeq} K(\C)$ is a homotopy equivalence. Combining our admissibility and d\'evissage conditions (see Definition \ref{admissible-def} and Definition \ref{devissage-cond-def}), we obtain our 
main d\'evissage theorem of this type.

\begin{theorem}[Single Type D\'{e}vissage] \label{devissage-1} 
Let $(\C, \A)$ be an admissible Waldhausen pair which satisfies the d\'{e}vissage condition. Then the inclusion $\A \hookrightarrow \C$ induces a homotopy equivalence $K(\A) \stackrel{\simeq}{\rightarrow} K(\C).$
\end{theorem}
\begin{proof}
Consider the following diagram
\begin{equation} \label{diagram}
\xymatrix{
K(\s^w_n \C_{\A}) \ar[r]  & K(\s_n \C_{\A}) \ar@/_/[d]^{\simeq} \ar[r] & K(\sgl_n \C_{\A}) \\
\prod_{1}^{n-1} K(\A) \ar[r]  \ar[u]_{K(\tau_n)} & \prod_{1}^n K(\A) \ar[r]^(.6){\vee} \ar@/_/[u] & K(\A). \ar[u]
}
\end{equation}
The middle vertical maps are homotopy equivalences from Proposition \ref{ho-eq-1}. The left vertical map was defined in Subsection \ref{basics-admissibility}. The left bottom map is induced by the exact functor 
$$(A_1, \cdots, A_{n-1}) \mapsto (A_1, A_2 \vee \Sigma A_1, \cdots, A_{n-1} \vee \Sigma A_{n-2}, \Sigma A_{n-1}).$$
The left square commutes up to homotopy by direct inspection. 

The bottom map $\prod_1^n K(\A) \to K(\A)$ is induced by the coproduct functor. The right vertical map is the canonical inclusion $\A = \sgl_1 \C_{\A} \to \sgl_n \C_{\A}$. The right square is also homotopy commutative, since the underlying exact functors are naturally weakly equivalent. Note that the bottom 
row defines a homotopy fiber sequence and the top row is a homotopy fiber sequence by Proposition \ref{fiber seq}. 

Passing to the colimit as $n \to \infty$, we obtain a homotopy commutative diagram 
\begin{equation} \label{diagram}
\xymatrix{
K(\s^w_{\infty} \C_{\A}) \ar[r]  & K(\s_{\infty} \C_{\A}) \ar@/_/[d]^{\simeq} \ar[r] & K(\sgl_{\infty} \C_{\A}) \\
\underrightarrow{\mathrm{colim}}_{n} \prod_{1}^{n-1} K(\A) \ar[r] \ar[u]^{\simeq}_{K(\tau_{\infty})} & \underrightarrow{\mathrm{colim}}_{n} \prod_{1}^n K(\A) \ar[r]^(.6){\vee} \ar@/_/[u] & K(\A) \ar[u]
}
\end{equation}
 whose the rows are homotopy fiber sequences (of infinite loop spaces) and $K(\tau_{\infty})$ is a homotopy equivalence because $(\C, \A)$ is admissible by assumption. 
Hence the right vertical map $K(\A) \to K(\sgl_{\infty} \C_{\A})$ is a homotopy equivalence. (Note that both of the right horizontal maps are $\pi_0$-surjective.)

By Proposition \ref{devissage-cond-prop}, the exact functor $ev_{\infty}$ induces a homotopy equivalence $K(\sgl_{\infty} \C_{\A}) \xrightarrow{\simeq} K(\C)$. The map in $K$-theory induced by the inclusion $\A \subset \C$ agrees with the composition of homotopy equivalences $K(\A) \xrightarrow{\simeq} K(\sgl_{\infty} \C_{\A}) \xrightarrow{\simeq} K(\C)$ and the result follows. 
\end{proof}

\begin{remark} \label{converse-devissage}
The proof of Theorem \ref{devissage-1} shows also the following converse statement: if $(\C, \A)$ satisfies the d\'evissage condition and the inclusion $\A \hookrightarrow \C$ induces a homotopy equivalence $K(\A) \stackrel{\simeq}{\to} K(\C)$, then the pair $(\C, \A)$ is admissible. Moreover,  Diagram \eqref{diagram} shows that $(\C, \A)$ is admissible if and only if the map $K(\A) \to K(\sgl_{\infty} \C_{\A})$ is a homotopy equivalence. 
\end{remark}

\subsection{D\'evissage functors} \label{devissage-functors}
The d\'evissage condition for a (good) Waldhausen pair $(\C, \A)$ is often a consequence of the existence of a \emph{d\'evissage functor} on $\C$ in the following sense.

\begin{definition}
A d\'evissage functor for $(\C, \A)$ is a (not necessarily exact!) 
functor
$$D: \C \to  E(\C, \C, \A), \ X \mapsto D(X) = (L(X) \rightarrowtail X \twoheadrightarrow A(X))$$
where $L: \C \to \C$ and $A: \C \to \A$ are (not necessarily exact!) functors such that:
\begin{itemize}
\item[(1)] $D$ preserves cofibrations,
\item[(2)] for each $X \in \C$, there is $n \geq 1$ such that $L^n(X) = \ast$. 
\end{itemize}
\end{definition}

\begin{proposition} \label{devissage-functor-prop}
Let $D$ be a d\'evissage functor for $(\C, \A)$. Then $(\C, \A)$ satisfies the d\'evissage condition. Moreover, if $(\C, \A)$ is also admissible, then the inclusion $\A \hookrightarrow \C$ induces a homotopy equivalence $K(\A) \stackrel{\simeq}{\to} K(\C)$. 
\end{proposition}
\begin{proof}
Given a cofibration $i \colon X \rightarrowtail Y$, we consider for $n$ large enough the morphism of filtered objects 
\[
\xymatrix{
\ast = L^{n+1}(X) \ar@{>->}[r] \ar@{=}[d] & L^n(X) \ar@{>->}[r] \ar@{>->}[d] & L^{n-1}(X) \ar@{>->}[r] \ar@{>->}[d] & \cdots \ar@{>->}[d] \ar@{>->}[r] & L(X) \ar@{>->}[r] \ar@{>->}[d] & X \ar@{>->}[d] \\
\ast = L^{n+1}(Y) \ar@{>->}[r]  & L^n(Y) \ar@{>->}[r]  & L^{n-1}(Y) \ar@{>->}[r] & \cdots \ar@{>->}[r] & L(Y) \ar@{>->}[r]  & Y. \\
}
\]
Then the objects $L^k(f)$, $1 \leq k \leq n$, defined by the pushouts 
\[
\xymatrix{
L^k(X) \ar@{>->}[r] \ar@{>->}[d] & X \ar@{>->}[d] \\
L^k(Y)  \ar@{>->}[r] & L^{k}(f) \\
}
\]
yield a factorization of $i \colon X \rightarrowtail Y$,
$$X \rightarrowtail L^{n}(f) \rightarrowtail L^{n-1}(f) \rightarrowtail \cdots \rightarrowtail L^1(f) \rightarrowtail L^0(f) := Y.$$
Here the morphisms are cofibrations in $\C$ because $D$ preserves cofibrations by assumption. The cofiber of $L^{k+1}(f) \rightarrowtail 
L^{k}(f)$, $k \geq 0$, is the cofiber of $A(L^{k}(X)) \rightarrowtail A(L^k(Y))$ which is in $\A$. Hence the d\'evisage condition
is satisfied for cofibrations. For an arbitrary morphism $f \colon X \to Y$, we choose a factorization $X \stackrel{i}{\rightarrowtail} Y' \stackrel{\sim}{\to} Y$ and find a factorization of $i$ as shown above. The second claim is a consequence of Theorem \ref{devissage-1}. 
\end{proof}

\begin{remark}(Exact d\'evissage functors)
As already suggested in the proof of Proposition \ref{devissage-functor-prop}, a d\'evissage functor $D$ for $(\C, \A)$ can be iterated in order to define an exact functor $D_{\infty} \colon \C \to \sgl_{\infty} \C_{\A}$ which is a section to $ev_{\infty}$. If, moreover, the d\'evissage functor $D$ is \emph{exact}, which is too strong an assumption in general, then $D_{\infty}$ actually defines an exact functor $D_{\infty} \colon \C \to \s_{\infty} \C_{\A}$. Assuming also that $D_{\infty}$ sends $\A$ to $\s_{\infty} \A$, then it can be shown that the map $K(\A) \xrightarrow{\simeq} K(\C)$ is a homotopy equivalence, with homotopy inverse induced by the composite exact functor:
$$\C \xrightarrow{D_{\infty}} \s_{\infty} \C_{\A} \xrightarrow{\underrightarrow{\colim}_n q_n} \underrightarrow{\colim}_n \ \A^n \xrightarrow{\vee} \A.$$

\end{remark}

\begin{example}(Relation to the Additivity Theorem) \label{additivity-example}
Let $\C$ be a good Waldhausen category with a cylinder functor that satisfies the cylinder axiom. We may consider the good Waldhausen pair that corresponds to the exact inclusion functor (cf. Theorem \ref{addit-thm}):
$$\C \times \C \hookrightarrow \s_2 \C, \ \ (X, Y) \mapsto (X \rightarrowtail X \vee Y \twoheadrightarrow Y).$$ 
There is a d\'evissage functor $D$ for $(\s_2 \C, \C \times \C)$ which is defined on objects by sending $(A \rightarrowtail C \twoheadrightarrow B)$ to 
the cofiber sequence in $E(\s_2 \C, \s_2 \C, \C \times \C)$:
$$D \colon (A \rightarrowtail C \twoheadrightarrow B) \mapsto \big((A = A \twoheadrightarrow \ast) \rightarrowtail (A \rightarrowtail C \twoheadrightarrow B) \twoheadrightarrow (\ast \rightarrowtail B = B)\big).$$
By Proposition \ref{devissage-functor-prop}, it follows that $(\s_2 \C, \C \times \C)$ satisfies the d\'evissage condition. According to the Additivity Theorem (Theorem \ref{addit-thm}), the map 
$$K(\C \times \C) \xrightarrow{\simeq} K(\s_2 \C)$$
is a homotopy equivalence. Equivalently, note that the d\'evissage functor $D$ is actually exact in this case, so the argument in the preceding remark applies. 

Then it follows from Remark \ref{converse-devissage} that the good Waldhausen pair $(\s_2 \C, \C \times \C)$ is admissible and therefore it satisfies the assumptions of Theorem \ref{devissage-1}. In this indirect way, we may view the statement of the Additivity Theorem for $\C$ as a special case of single type d\'evissage (Theorem \ref{devissage-1}) -- whose proof, of course, made essential use of the Additivity Theorem.   
\end{example}

\subsection{Example: Abelian categories} \label{abelian-example}

Let $\mathcal C$ be a (small) abelian category. This may be regarded as a Waldhausen category in the standard way by defining the cofibrations to be the monomorphisms and the weak equivalences to be the isomorphisms. Let $\Ch(\mathcal C)$ denote the Waldhausen category of bounded chain complexes in 
$\mathcal C$, where the cofibrations are the monomorphisms and the weak equivalences are the quasi-isomorphisms of chain complexes. According to the Gillet--Waldhausen theorem \cite{TT}, the exact inclusion functor $\mathcal C \to \Ch(\mathcal C)$, as chain complexes concentrated in degree $0$, induces a homotopy equivalence in $K$-theory 
\begin{equation} \label{Gillet-Waldhausen-eq} 
K(\mathcal C) \xrightarrow{\simeq} K(\Ch(\mathcal C)).
\end{equation}

Let $\mathcal A \subset \mathcal C$ be a full exact abelian subcategory. The corresponding inclusion of chain complexes $\Ch(\mathcal A) \subset \Ch(\mathcal C)$ defines a Waldhausen subcategory. In addition, $(\Ch(\cab), \Ch(\a))$ is a good Waldhausen pair using the standard cylinder objects for 
chain complexes. We emphasize that the exact inclusion $\mathrm{Ch}^b(\a) \subset \mathrm{Ch}^b(\cab)$ is full but not homotopically full in general.

\medskip

We recall Quillen's d\'evissage theorem for the $K$-theory of abelian categories. 

\begin{theorem}[Quillen \cite{Qu}, D\'{e}vissage] \label{quillen}
Let $\mathcal C$ be an abelian category and let $\mathcal A \subset \mathcal C$ be a full exact abelian subcategory which is closed in $\mathcal C$ under subobjects and quotients. Suppose that every object $C \in \mathcal C$ admits a finite 
filtration 
$$0=C_0 \subseteq C_1 \subseteq C_2 \subseteq \cdots \subseteq C_{n-1} \subseteq C_n = C$$
such that $C_i / C_{i-1} \in \mathcal A$ for all $i = 1, \cdots, n$. Then the inclusion $\mathcal A \hookrightarrow \mathcal C$ induces a homotopy 
equivalence $K(\mathcal A) \stackrel{\simeq}{\to} K(\mathcal C)$.
\end{theorem}

We refer to \cite{CZ} for an extension of this d\'evissage theorem to the $K$-theory of ACGW-categories and applications in this context.
  
The purpose of this subsection is to relate Quillen's d\'evissage theorem to Theorem \ref{devissage-1} when applied to the respective categories of bounded chain complexes. The d\'{e}vissage condition in Quillen's theorem is seemingly weaker than the d\'{e}vissage condition of Definition \ref{devissage-cond-def} -- while Theorem \ref{quillen} requires the existence of appropriate filtrations of objects, the d\'{e}vissage condition of Definition \ref{devissage-cond-def} essentially requires the existence of such factorizations for all maps of chain complexes. However, it turns out that the d\'evissage condition for $(\Ch(\cab), \Ch(\a))$ is satisfied under the assumptions of Quillen's theorem.

\begin{lemma} \label{strong devissage condition}
Let $\a \subset \cab$ be as in Theorem \ref{quillen}. Then for every $C \in \cab$ and subobject $C' \subseteq C$, there is a filtration 
$$C' = C'_0 \subseteq C'_1 \subseteq C'_2 \subseteq \cdots \subseteq C'_{n-1} \subseteq C'_n = C$$
such that $C'_i / C'_{i-1} \in \a$ for all $i=1, \cdots, n$. 
\end{lemma}
\begin{proof}
There is a filtration $0 = C_0 \subseteq C_1 \subseteq C_2 \subseteq \cdots \subseteq C_{n-1} \subseteq C_n = C$
such that $C_i / C_{i-1} \in \a$ for $i \geq 1$. For $i \geq 0$, we consider the pushouts of subobjects of $C$:
\[
 \xymatrix{
 C' \cap C_i \ar[r] \ar[d] & C_i \ar[d] \\
 C' \ar[r] & C'_i. 
 }
\]
Thus, there is a filtration $C' \subseteq C'_1 \subseteq C'_2 \subseteq \cdots \subseteq C'_{n-1} \subseteq C'_n = C$ and the cokernels $C'_i / C'_{i-1}$ 
are quotients of $C_i / C_{i-1}$, hence they are again in $\a$.  
\end{proof}

\begin{proposition}  \label{quillen-devissage-cond}
Let $\a \subset \cab$ be as in Theorem \ref{quillen}. Then $(\Ch(\cab), \Ch(\a))$ satisfies the d\'evissage condition. 
\end{proposition}
\begin{proof}
Consider an object in $\Ch(\cab)$ 
$$C_{\bullet} = ( \cdots \to 0 \to C_n \to C_{n-1} \to \cdots \to C_1 \to C_0 \to 0 \to \cdots)$$
and a subobject $C'_{\bullet} \hookrightarrow C_{\bullet}$. First, by Lemma \ref{strong devissage condition}, we may suppose that there is a filtration of $C'_n \subseteq C_n$:
$$C'_n = C_{0,n} \subseteq C_{1,n} \subseteq \cdots \subseteq C_{m-1,n} \subseteq C_{m,n} = C_n$$
whose successive subquotients lie in $\a$. Set $X_{0, n-1} : = C'_{n-1}$. For $i=1, \cdots, m$, define $X_{i, n-1}$ inductively by pushout squares
$$
\xymatrix{
C_{i-1, n} \ar[r] \ar[d] & C_{i, n} \ar[d] \\
X_{i-1, n-1} \ar[r] & X_{i, n-1}.
}
$$
(Note that $C_{0, n} = C'_n \xrightarrow{\partial} C'_{n-1} = X_{0, n-1}$.) We obtain in this way a factorization of the inclusion $C'_{n-1} \subseteq C_{n-1}$ as follows:
$$C'_{n-1} = X_{0,n-1} \subseteq X_{1,n-1} \subseteq \cdots \subseteq X_{m-1,n-1} \subseteq X_{m,n-1} \to C_{n-1}$$
which has the required property except possibly at the last stage. Let $\overline{C}_{i, n-1}$ denote the image of 
$X_{i, n-1}$ in $C_{n-1}$. Then there is a filtration of $C'_{n-1} \subseteq C_{n-1}$
\begin{equation} \label{filtration-1}
C'_{n-1} = \overline{C}_{0,n-1} \subseteq \overline{C}_{1,n-1} \subseteq \cdots \subseteq \overline{C}_{m-1,n-1} \subseteq \overline{C}_{m,n-1} \subseteq C_{n-1}
\end{equation}
whose successive subquotients, except possibly for the last one, are in $\a$, since $\a$ is closed under taking quotients in $\cab$. By Lemma \ref{strong devissage condition}, there is a further filtration
\begin{equation} \label{filtration-2} 
\overline{C}_{m, n-1} = C_{0, n-1} \subseteq C_{1, n-1} \subseteq \cdots \subseteq C_{m', n-1} = C_{n-1} 
\end{equation}
whose successive subquotients are in $\a$. Joining these two filtrations \eqref{filtration-1} and \eqref{filtration-2}, we obtain a filtration of 
$C'_{n-1} \subseteq C_{n-1}$ with the required property that its subquotients are in $\a$. Combined with the original filtration of $C'_n \subseteq C_n$, we obtain a filtration for the inclusion: 
$$(C'_n \stackrel{\partial}{\to} C'_{n-1}) \hookrightarrow (C_n \stackrel{\partial}{\to} C_{n-1}).$$
Repeating this process inductively on the length of the chain complex, we obtain a filtration for the inclusion $C'_{\bullet} \subseteq C_{\bullet}$ as required. 
This shows that the d\'evissage condition is satisfied for cofibrations between chain complexes. For an arbitrary chain map $f \colon C'_{\bullet} \to C_{\bullet}$, we choose a factorization $C'_{\bullet} \stackrel{i}{\rightarrowtail} C''_{\bullet} \stackrel{\sim}{\to} C_{\bullet}$ and apply the construction above to the 
cofibration $i \colon C'_{\bullet} \subseteq C''_{\bullet}$. 
\end{proof}

\begin{remark}
Let $\Ch(\cab)^{\mathrm{ac}}$ (resp. $\Ch(\a)^{\mathrm{ac}}$) denote the full Waldhausen subcategory which is spanned by the acyclic chain complexes. Then the Waldhausen pair $(\Ch(\cab)^{\mathrm{ac}}, \Ch(\a)^{\mathrm{ac}})$ is again good. In addition, the factorizations constructed in Proposition \ref{quillen-devissage-cond} apply also to the Waldhausen pair $(\Ch(\cab)^{\mathrm{ac}}, \Ch(\a)^{\mathrm{ac}})$. Hence this Waldhausen pair satisfies the d\'evissage condition, too. 
\end{remark}

Let $\a \subset \cab$ be as in Theorem \ref{quillen}. Applying Theorem \ref{quillen} and the Gillet--Waldhausen homotopy equivalence \eqref{Gillet-Waldhausen-eq}, we conclude that the exact inclusion functor $\Ch(\a) \hookrightarrow \Ch(\cab)$ induces a homotopy equivalence:
$$K(\Ch(\a)) \stackrel{\simeq}{\longrightarrow} K(\Ch(\cab)).$$
Then it follows from Proposition \ref{quillen-devissage-cond} and Remark \ref{converse-devissage} that the good Waldhausen pair $(\Ch(\cab), \Ch(\a))$ is admissible -- and therefore it satisfies the assumptions of single type d\'evissage (Theorem \ref{devissage-1}). 

Clearly it would be desirable to prove the admissibility of $(\Ch(\cab), \Ch(\a))$ independently of Quillen's d\'evissage theorem (Theorem \ref{quillen}), and therefore deduce Theorem \ref{quillen} from the d\'evissage theorem of single type (Theorem \ref{devissage-1}). More generally, it  would be interesting to know whether the good Waldhausen pair $(\Ch(\cab), \Ch(\a))$ is admissible for general full exact abelian subcategories $\a \subset \cab$, independently of the d\'evissage condition.

\section{Multiple Type D\'evissage}

\noindent \textbf{Assumptions.} We fix the following notation and assumptions throughout this section. Let $\C$ be a good Waldhausen category equipped 
with a cylinder functor which satisfies the cylinder axiom. Let $\A = (\A_i)_{i \geq 1}$ be a collection of Waldhausen subcategories such that for every $i \geq 1$:
\begin{itemize}
\item[(a)] $(\C, \A_i)$ is a replete Waldhausen pair (Definition \ref{replete-def}), 
\item[(b)] The Waldhausen subcategory $\A_i$ is closed under extensions in $\C$, i.e., given a cofiber sequence $A \rightarrowtail X \twoheadrightarrow B$ in $\C$ with $A, B \in \A_i$, then $X \in \A_i$,
\item[(c)] The suspension functor $\Sigma \colon \C \to \C$ sends $\A_i$ to $\A_{i+1}$.
\end{itemize}
We note that the Waldhausen category $\sgl_{\infty} \C_{\A}$ is good by Lemma \ref{additional-props}.
 
\subsection{Admissibility} \label{admissibility-2-subsec}
As in the case of d\'evissage of single type, the purpose of an admissibility assumption on $(\C, \A)$ is to identify the $K$-theory of the Waldhausen category 
$\s^w_{\infty} \C_{\A}$.  In the multiple type case, the relevant admissibility assumption is much stronger, however, at the same time, it is significantly easier to verify in many examples of interest. 

\begin{definition} \label{admissible-2}  We say that $(\C, \A)$ is \emph{admissible} if for every $X_{\bullet, \bullet} \in \s^w_{\infty} \C_{\A}$, we have that $X_{0, k} \in \A_k$ for every $k \geq 1$. 
\end{definition}

This admissibility property is motivated by the following example from \cite{Wa}.

\begin{example} \label{admissible-spherical-objs-example}
Let $\C$ be a Waldhausen category associated with a homology theory in the sense of \cite[1.7]{Wa}, and let $\A_i$ denote the Waldhausen subcategory of spherical objects of dimension $i-1$. Then $(\C, \A = (\A_i)_{i \geq 1})$ is admissible by \cite[Lemma 1.7.4]{Wa}. 
\end{example}

\noindent For each $n \geq 1$, we have an exact functor (cf. Subsection \ref{basics-admissibility}):
$$q^w_n \colon \s^w_n \C_{\A} \to \A_1 \times \cdots \times \A_{n-1}, \ \ X_{\bullet, \bullet} \mapsto (X_{0,1}, \cdots, X_{n-2, n-1}).$$
Assuming that $(\C, \A)$ is admissible, we also have a well-defined exact functor:
$$p_n \colon \s^w_n \C_{\A} \to \A_1 \times \cdots \times \A_{n-1}, \ \ X_{\bullet, \bullet} \mapsto (X_{0,1}, \cdots, X_{0, n-1}).$$
Similarly to Subsection \ref{basics-admissibility}, there are also exact functors $\tau_n \colon \prod_1^{n-1} \A_i \to \s^w_n \C_{\A}$ for $n \geq 2$. In detail,  $\tau_n$ is defined on objects by 
$$\tau_n \colon (A_1, \cdots, A_{n-1}) \mapsto (\ast \rightarrowtail A_1 \rightarrowtail A_2 \vee CA_1 \rightarrowtail \cdots \rightarrowtail A_{n-1} \vee \bigvee_1^{n-2} CA_{i} \rightarrowtail \bigvee_1^{n-1} CA_i).$$
(Note that this functor is well defined because of assumption (c) on $(\C, \A)$.) Moreover, the following diagram of exact functors commutes
$$
\xymatrix{
\s^w_n \C_{\A} \ar[r]^{i^w_{n, \A}} & \s^w_{n+1} \C_{\A} \\
\prod_1^{n-1} \A_i \ar[u]_{\tau_n} \ar[r] & \prod_1^n \A_i \ar[u]_{\tau_{n+1}} 
}
$$ 
where the bottom functor is the canonical inclusion functor, given on objects by $(A_1, \cdots, A_{n-1}) \mapsto (A_1, \cdots, A_{n-1}, \ast)$. As a consequence, we also obtain an exact functor $\tau_{\infty} \colon \underrightarrow{\colim}_n \prod_1^{n-1} \A_i \longrightarrow \s^w_{\infty} \C_{\A}.$

\medskip

The composite functor $p_n \circ \tau_n$ is weakly equivalent to the identity functor.  The composite functor $q^w_n \circ \tau_n$ is given on objects by 
$$(A_1, \cdots, A_{n-1}) \mapsto (A_1, A_2 \vee \Sigma A_1, \cdots, A_{n-1} \vee \Sigma A_{n-2})$$
and it induces a homotopy equivalence in $K$-theory. More specifically, the map in $K$-theory is identified with the homotopy equivalence:
$$ (\pi_1, \pi_2 + \Sigma \circ \pi_1, \cdots, \pi_{n-1} + \Sigma \circ \pi_{n-2}): \prod_1^{n-1} K(\A_i) \to \prod_1^{n-1} K(\A_i)$$
where $\pi_i$ denotes the projection onto the $i$-th factor, the sum corresponds to the loop sum, and $\Sigma \colon K(\A_i) \to K(\A_{i+1})$ denotes here, by a slight abuse of notation, the map that is induced by the suspension functor on $\C$. Each map $K(p_n)$, $K(q^w_n)$ and $K(\tau_n)$ is a homotopy equivalence if one of them is a homotopy equivalence. 

\begin{proposition} \label{admissible2-prop}
Suppose that $(\C, \A)$ be admissible. Then the map 
$$K(p_n) \colon K(\s^w_n \C_{\A}) \to K(\A_1) \times \cdots \times K(\A_{n-1})$$
is a homotopy equivalence.
\end{proposition}
\begin{proof}
The proof is essentially the same as the proof of \cite[Lemma 1.7.5]{Wa} and will be omitted. 
\end{proof}

\subsection{The d\'evissage condition} \label{devissage-cond-subsec-2} First we define an abstract notion of connectivity in $\C$ with respect to $\A$. This definition is inspired by the \emph{Hypothesis} in \cite[1.7, p. 361]{Wa}.

\begin{definition}
A morphism $f \colon X \to Y$ in $\C$ is \emph{k--connected (with respect to $\A$)}, $k \geq -1$, if there is a weak equivalence $g \colon Y \xrightarrow{\sim} Y'$ and a factorization of $g f \colon X \to Y'$,
$$X = X_0 \rightarrowtail X_1 \rightarrowtail X_2 \rightarrowtail \cdots \rightarrowtail X_m \xrightarrow{\sim} Y',$$
such that $X_i/X_{i-1} \in \A_{k + i + 1}$ for every $i \geq 1$. 
\end{definition}

\begin{definition}  \label{cancellation-def}
We say that $(\C, \A)$ has the \emph{cancellation property} if the following holds: for any $k \geq -1$ and any composable morphisms $X \xrightarrow{f} Y \xrightarrow{g} Z$ in $\C$ such that $f$ and $gf$ are $k$-connected, then $g$ is also $k$-connected. 
\end{definition}

\begin{example} \label{admissibility-spherical-objes-example}
The meaning of the cancellation property may be unclear directly from the abstract definition of $\A$-connectivity, but it can be easily verified in many examples of interest in which $\A$-connectivity corresponds to a standard notion of connectivity. For example, the cancellation property is obviously 
satisfied in the context of Waldhausen's theorem on spherical objects in \cite[1.7, pp. 360--361]{Wa}. 
\end{example}

As in the case of d\'evissage of single type, the main condition on $(\C, \A)$ that we are interested in is the following analogue of Definition 
\ref{devissage-cond-def}. 

\begin{definition} \label{devissage-cond-def2}
We say that $(\C, \A)$ satisfies the \emph{d\'{e}vissage condition} if for every morphism $f: X \rightarrow Y$ in $\C$, there is a weak equivalence 
$g: Y \stackrel{\sim}{\rightarrow} Y'$ such that the composition $gf: X \to Y'$ 
admits a factorization
  $$X = X_0 \rightarrowtail X_1 \rightarrowtail \cdots \rightarrowtail X_{m-1} \rightarrowtail X_m \stackrel{\sim}{\to} Y'$$
where $X_i / X_{i-1} \in \A_i$ for every $i \geq 1$.
\end{definition}

\begin{remark}
The d\'evissage condition for $(\C, \A)$ exactly says that every morphism in $\C$ is (--1)-connected with respect to $\A$. We emphasize that the factorizations in Definition \ref{devissage-cond-def2} are not assumed to be functorial or invariant in any sense. However, as we will see in 
Remark \ref{uniqueness-filtrations-2}, this condition is sufficient in order to conclude that the functor $ev_{\infty}$ (defined below) is an equivalence of homotopy theories. 
\end{remark}

Similarly to the single type case, there are exact functors $ev_n: \sgl_n \C_{\A} \to \C$ for $n \geq 1$, given by $(X_{\bullet, \bullet}) \mapsto X_{0,n},$
which are compatible with respect to stabilization along the inclusion functors $\widehat{i}_{n, \A}$. Thus, we obtain an induced exact functor 
$$ev_{\infty}: \sgl_{\infty} \C_{\A} \to \C$$
which sends a staircase diagram representing a bounded $\A$-filtration of an object $X_{0,n}$, for $n$ large enough, to the object $X_{0,n}$ itself. As in the single type case, the main function of the d\'evissage condition will be to ensure that the functor $ev_{\infty}$ induces a homotopy equivalence in $K$-theory. 

\begin{proposition} \label{devissage-cond-prop2}
Suppose that $(\C, \A)$ satisfies the d\'evissage condition and has the cancellation property. Then the functor $ev_{\infty} \colon \sgl_{\infty} \C_{\A} \to \C$ induces a homotopy equivalence $K(\sgl_{\infty} \C_{\A}) \stackrel{\simeq}{\to} K(\C)$.
\end{proposition}
\begin{proof}
The proof is similar to \cite[Lemma 1.7.2]{Wa}. We show that the exact functor $ev_{\infty} \colon \sgl_{\infty} \C_{\A} \to \C$ satisfies the assumptions of the Approximation Theorem (Theorem \ref{approx-thm}). (App1) holds by definition. For (App2), we consider a morphism $ev_{\infty}(X_{\bullet, \bullet}) = X_{0,n} \to Y$ where $X_{\bullet, \bullet} \in \sgl_n \C_{\A}$ and proceed by induction on $n \geq 0$ (where $\sgl_0 \C_{\A} = \s_0 \C = \{\ast\}$). (App2) holds for $n = 0$ by the d\'evissage condition applied to the morphism $\ast = X_{0,0} \to Y$. Suppose by induction 
that (App2) holds for morphisms $ev_{\infty}(X_{\bullet, \bullet}) \to Y$ where $X_{\bullet, \bullet} \in \sgl_{n-1} \C_{\A}$. Now consider a morphism as follows (we ignore the cofibers for simplicity):
$$(\ast \rightarrowtail X_{0,1} \rightarrowtail \cdots \rightarrowtail X_{0,n}) \in \sgl_n \C_{\A}, \ \ f \colon ev_{\infty}(X_{\bullet, \bullet}) = X_{0,n} \to Y.$$  
By induction, there is an object $X'_{\bullet, \bullet} \in \sgl_k \C_{\A}$ and a morphism in $\sgl_k \C_{\A}$, for some $k \geq n-1$,  
$$
\xymatrix{
\ast \ar@{>->}[r] \ar@{=}[d] & X_{0,1} \ar@{>->}[r] \ar[d] & \cdots \ar@{>->}[r] & X_{0, n-1} \ar[d] \ar@{=}[r] & \cdots \ar[d] \ar@{=}[r] & X_{0, n-1} \ar[d] \\
 \ast \ar@{>->}[r] & X'_{0,1} \ar@{>->}[r]  & \cdots \ar@{>->}[r] & X'_{0, n-1} \ar@{>->}[r] & \cdots \ar@{>->}[r] & X'_{0,k} 
}
$$
together with weak equivalences $g \colon Y \xrightarrow{\sim} Y'$ and $f' \colon X'_{0, k} \xrightarrow{\sim} Y'$ such that the following square in $\C$ commutes 
\begin{equation} \label{square-1}
\xymatrix{
X_{0,n-1} \ar@{>->}[r] \ar[d] & X_{0,n} \ar[r]^f & Y \ar[dd]_{\sim}^g \\
X'_{0, n-1} \ar[d] & \\
X'_{0,k} \ar[rr]^{\sim}_{f'} && Y'.
}
\end{equation}
Then, by definition, the composition $X'_{0,\ell} \to X'_{0, k} \xrightarrow{\sim} Y'$, $1 \leq \ell  \leq k$, is  $(\ell-1)$-connected. The canonical morphism 
$$X'_{0, n-1} \to X'_{0,n-1} \cup_{X_{0, n-1}} X_{0,n}$$
is $(n-2)$-connected because its cofiber is in $\A_n$. It follows by the cancellation property that the morphism induced by \eqref{square-1}: 
$$u \colon X'_{0,n-1} \cup_{X_{0, n-1}} X_{0,n} \to Y'$$
is also $(n-2)$-connected. Therefore there is a weak equivalence $g' \colon Y' \xrightarrow{\sim} Y''$ and a factorization of $g'u$,
$$X'_{0,n-1} \cup_{X_{0, n-1}} X_{0,n} = Z_0 \rightarrowtail Z_1 \rightarrowtail \cdots \rightarrowtail Z_m \stackrel{\sim}{\to} Y'',$$
where $Z_{i}/Z_{i-1} \in \A_{n + i -1}$ for every $i \geq 1$.  Note that there is a cofiber sequence
$$X_{n-1, n} = X_{0,n}/X_{0, n-1} \rightarrowtail Z_1/X'_{0,n-1} \twoheadrightarrow Z_1/(X'_{0,n-1} \cup_{X_{0,n-1}} X_{0,n}).$$
Since $\A_n$ is closed under extensions in $\C$, it follows that $Z_1/X'_{0,n-1} \in \A_n$. Thus, we obtain a morphism in $\sgl_{\infty} \C_{\A}$ represented by the diagram:
$$
\xymatrix{
\ast \ar@{>->}[r] \ar@{=}[d] & X_{0,1} \ar@{>->}[r] \ar[d] & \cdots \ar@{>->}[r] & X_{0, n-1} \ar[d] \ar@{>->}[r] & X_{0,n} \ar[d] \ar@{=}[r] & \cdots \ar@{=}[r] & X_{0,n} \ar[d] \\
\ast \ar@{>->}[r] & X'_{0,1} \ar@{>->}[r]  & \cdots \ar@{>->}[r] & X'_{0, n-1} \ar@{>->}[r] & Z_1 \ar@{>->}[r] & \cdots \ar@{>->}[r] & Z_m 
}
$$
together with a commutative square in $\C$:
$$
\xymatrix{
X_{0,n} \ar[r]^f \ar[dd] & Y \ar[d]_{\sim}^g \\
& Y' \ar[d]_{\sim}^{g'} \\
Z_m \ar[r]^{\sim} & Y''.
}
$$ 
This completes the inductive proof that $ev_{\infty}$ satisfies (App2). Then the result follows as an application of Theorem \ref{approx-thm}. 
\end{proof}

\begin{remark} \label{uniqueness-filtrations-2}
As a consequence of the results stated in Remark \ref{remark-approx}, the proof of Proposition \ref{devissage-cond-prop2} also shows that the exact functor 
$ev_{\infty} \colon \sgl_{\infty} \C_{\A} \to \C$ induces an equivalence between the underlying homotopy theories. 
\end{remark}

\subsection{The main theorem} \label{main-thm-2}
We recall that the suspension functor $\Sigma \colon \C \to \C$ restricts to exact functors $\Sigma \colon \A_i \to \A_{i+1}$, $i \geq 1$, by assumption. Thus the collection of maps $K(\A_i) \to K(\C)$, $i \geq 1$, induces canonically a map:
\begin{equation} \label{multiple-type-map}
\underrightarrow{\hocolim}_{(\Sigma)} K(\A_i) \longrightarrow \underrightarrow{\hocolim}_{(\Sigma)} K(\C) \simeq K(\C)
\end{equation}
where the last homotopy equivalence holds because $\Sigma$ induces a homotopy equivalence in $K$-theory. A d\'evissage theorem of multiple type is a statement that for certain $(\C, \A)$ for which a d\'evissage--type condition is satisfied, the map \eqref{multiple-type-map} is a homotopy equivalence. The following result is an abstract version and 
a generalization of Waldhausen's theorem on spherical objects in \cite[1.7]{Wa}. 

\begin{theorem}[Multiple Type D\'evissage] \label{devissage-multiple-type}
Let $(\C, \A)$ be admissible and suppose that it has the cancellation property and satisfies the d\'evissage condition. Then the canonical map 
$$\underrightarrow{\hocolim}_{(\Sigma)} K(\A_i) \stackrel{\simeq}{\longrightarrow} \underrightarrow{\hocolim}_{(\Sigma)} K(\C) \simeq K(\C)$$
is a homotopy equivalence. 
\end{theorem}
\begin{proof} We start by considering the properties of the following diagram: 
\begin{equation} \label{diagram-2}
\xymatrix{
K(\s^w_n \C_{\A}) \ar[r]  & K(\s_n \C_{\A}) \ar[r] & K(\sgl_n \C_{\A}) \\
\prod_{1}^{n-1} K(\A_i) \ar[r]  \ar[u]_{K(\tau_n)}^{\simeq} & \prod_{1}^n K(\A_i) \ar[r]^(.6){\pi_n} \ar[u]_{K(\tau'_n)}^{\simeq} & K(\A_n) \ar[u]
}
\end{equation}
The left vertical map was defined in Subsection \ref{admissibility-2-subsec}. It is a homotopy equivalence by Proposition \ref{admissible2-prop}. 
The middle map is induced by an exact functor $\tau'_n \colon \prod_1^n \A_i \to \s_n \C_{\A}$, a variation of the functor $\tau_n$, given on objects by
$$(A_1, A_2, \cdots, A_n) \mapsto (\ast \rightarrowtail A_1 \rightarrowtail A_2 \vee CA_1 \rightarrowtail \cdots \rightarrowtail A_n \vee \bigvee_1^{n-1} CA_i).$$
Note that the composition of $\tau'_n$ with the functor $q_n \colon \s_n \C_{\A} \to \prod_1^n \A_i$ (see Proposition \ref{ho-eq-1}) is the exact functor
$$\prod_1^n \A_i \to \prod_1^n \A_i, \ \ (A_1, \cdots, A_n) \mapsto (A_1, A_2 \vee \Sigma A_1, \cdots, A_n \vee \Sigma A_{n-1}).$$
This last functor induces a homotopy equivalence in $K$-theory. Since $K(q_n)$ is also a homotopy equivalence by Proposition \ref{ho-eq-1}, it follows that the middle vertical map $K(\tau'_n)$ in \eqref{diagram-2} is also a homotopy equivalence.

The left bottom map is induced by the inclusion functor 
$$(A_1, \cdots, A_{n-1}) \mapsto (A_1, A_2, \cdots, A_{n-1}, \ast).$$
Then the left square in \eqref{diagram-2} commutes by direct inspection. (On the other hand, note that the functors $\{\tau'_n\}_{n \geq 1}$ are not natural with respect to  the functors $\{i_{n, \A}\}_{n \geq 1}$ and the inclusion functors $\prod_{1}^n \A_i \to \prod_{1}^{n+1} \A_i$ as defined above; however, we can easily identify the appropriate map $\prod_{1}^n K(\A_i) \to \prod_1^{n+1} K(\A_i)$, so that the maps $\{K(\tau'_n)\}_{n \geq 1}$ become compatible for all $n$.)
 
The bottom map $\pi_n \colon \prod_1^n K(\A_i) \to K(\A_n)$ is the projection onto the last factor. The right vertical map is induced by the exact inclusion functor $\A_n \to \sgl_n \C_{\A}$ which is given on objects by $$A \mapsto (\ast =  \cdots  = \ast \rightarrowtail A).$$ The right square in 
\eqref{diagram-2} is homotopy commutative, since the underlying exact functors are naturally weakly equivalent. 

Note that the bottom row in \eqref{diagram-2} defines a homotopy fiber sequence.

\medskip

Let $\sigma \colon \s_n \C_{\A} \to \s_{n+1} \C_{\A}$ be the exact functor that is given on objects by 
$$(\ast \rightarrowtail X_{0,1} \rightarrowtail \cdots \rightarrowtail X_{0,n}) \mapsto (\ast  = \ast \rightarrowtail \Sigma X_{0,1} \rightarrowtail \cdots \rightarrowtail \Sigma X_{0,n})$$
and let $\sigma' \colon \prod_1^n \A_i \to \prod_1^{n+1} \A_i$ be the exact functor given by $(A_1, \cdots, A_n) \mapsto (\ast, \Sigma A_1, \cdots, \Sigma A_n)$. The maps in Diagram \eqref{diagram-2} are compatible with the functors $\sigma$ and $\sigma'$. Passing in \eqref{diagram-2} to the homotopy colimits along these stabilization functors as $n \to \infty$, we obtain a homotopy commutative diagram: 
$$
\xymatrix{
\underrightarrow{\hocolim}_{(\sigma)} K(\s^w_n \C_{\A}) \ar[r]  & \underrightarrow{\hocolim}_{(\sigma)} K(\s_n \C_{\A}) \ar[r] & \underrightarrow{\hocolim}_{(\sigma)} K(\sgl_n \C_{\A}) \\
\underrightarrow{\hocolim}_{(\sigma')} \prod_{1}^{n-1} K(\A_i) \ar[r]  \ar[u]^{\simeq} & \underrightarrow{\hocolim}_{(\sigma')} \prod_{1}^n K(\A_i) \ar[r] \ar[u]^{\simeq} & \underrightarrow{\hocolim}_{(\Sigma)} K(\A_n). \ar[u]
}
$$
Note that the bottom row defines a homotopy fiber sequence. An application of the Additivity Theorem (Theorem \ref{addit-thm-2}) shows that the map 
$$K(\sigma) \colon K(\sgl_n \C_{\A}) \to K(\sgl_{n+1} \C_{\A})$$ 
agrees 
with the stabilization map $K(\widehat{i}_{n, \A})$ up to sign -- this uses the definition of the suspension functor in $\sgl_{\infty} \C_{\A}$, see Lemma \ref{additional-props}. 

We claim that the top row in the last diagram also defines a homotopy fiber sequence. To see this, consider the canonical map from the top row of the diagram to the respective sequence of maps for $n = \infty$:
$$
\xymatrix{
\underrightarrow{\hocolim}_{(\sigma)} K(\s^w_n \C_{\A}) \ar[r] \ar[d] & \underrightarrow{\hocolim}_{(\sigma)} K(\s_n \C_{\A}) \ar[r] \ar[d] & \underrightarrow{\hocolim}_{(\sigma)} K(\sgl_n \C_{\A}) \ar[d]^{\simeq} \\
\underrightarrow{\hocolim}_{(\Sigma)} K(\s^w_{\infty} \C_{\A}) \ar[r] & \underrightarrow{\hocolim}_{(\Sigma)} K(\s_{\infty} \C_{\A}) \ar[r] & \underrightarrow{\hocolim}_{(\Sigma)} K(\sgl_{\infty} \C_{\A}). 
}
$$
Here the functor $\Sigma$ denotes the exact functor which defines the suspension functor on $\sgl_{\infty} \C_{\A}$ and restricts to the functor $\sigma$ as defined above (cf. the proof of Lemma \ref{additional-props}). The previous observations about $K(\sigma)$ show that the right vertical map is a homotopy equivalence (as indicated in the diagram). The bottom row is a homotopy 
fiber sequence (of infinite loop spaces) by Proposition \ref{fiber seq}. In addition, it can be verified that the left square is a homotopy pullback using the identifications in Diagram \eqref{diagram-2} to determine and identify the horizontal (co)fibers. As a consequence, the top row of the diagram is also a homotopy fiber sequence. 

\smallskip

Returning now to the previous diagram of homotopy fiber sequences above, it follows that the right vertical map,
$$\underrightarrow{\hocolim}_{(\Sigma)} K(\A_n) \to  \underrightarrow{\hocolim}_{(\sigma)} K(\sgl_n \C_{\A}),$$
is a homotopy equivalence. Next we consider the commutative diagrams
$$
\xymatrix{
K(\A_n) \ar[d]^{\Sigma} \ar[r] &  K(\sgl_n \C_{\A}) \ar[d]^{K(\sigma)} \ar[rr]^{K(ev_n)} && K(\C) \ar[d]^{\Sigma} \\
K(\A_{n+1}) \ar[r] & K(\sgl_{n+1} \C_{\A}) \ar[rr]^(.6){K(ev_{n+1})} && K(\C)  
}
$$
and passing to the homotopy colimits as $n \to \infty$, we obtain the following maps:
\begin{equation} \label{comp-map}
\underrightarrow{\hocolim}_{(\Sigma)} K(\A_n) \xrightarrow{\simeq} \underrightarrow{\hocolim}_{(\sigma)} K(\sgl_n \C_{\A}) \to \underrightarrow{\hocolim}_{(\Sigma)} K(\C) \simeq K(\C)
\end{equation}
The last map $\underrightarrow{\hocolim}_{(\sigma)} K(\sgl_n \C_{\A}) \to K(\C)$ can be identified with $K(ev_{\infty})$ and therefore it is a homotopy equivalence by Proposition \ref{devissage-cond-prop2}. Then the composite map \eqref{comp-map} is a homotopy equivalence, as required. 
\end{proof}

\begin{remark}
Note that the cancellation property and the d\'evissage condition were needed in the proof of Theorem \ref{devissage-multiple-type} exactly in order to conclude that the functor $ev_{\infty}$ induces a $K$-equivalence (Proposition \ref{devissage-cond-prop2}).
\end{remark}

\begin{example}(Waldhausen's theorem on spherical objects \cite{Wa}) \label{Waldhausen-theorem}
Let $(\C, \A)$ be as in Example \ref{admissible-spherical-objs-example}. Under the assumption of \cite[Hypothesis, p. 361]{Wa}, $(\C, \A)$ has the cancellation property and satisfies the d\'evissage condition. Thus we recover Waldhausen's theorem on spherical objects \cite[Thorem 1.7.1]{Wa} 
as a special case of Theorem \ref{devissage-multiple-type}. 
\end{example}

\begin{example}($\mathbb{Z}$-grading and the Gillet--Waldhausen theorem) \label{Gillet-Waldhausen}
Some important examples of d\'evissage--type theorems of multiple type in the stable context involve collections of Waldhausen subcategories $\A = (\A_i)$ for $i \in \mathbb{Z}$ (e.g., the Gillet--Waldhausen theorem \cite{TT} or Example \ref{theorem_of_the_heart} below). 

In these cases, $\A_i$ is typically the full Waldhausen subcategory of objects which are concentrated in degree $i$ (appropriately interpreted), and the maps $$\Sigma \colon K(\A_i) \xrightarrow{\simeq} K(\A_{i+1})$$ are homotopy equivalences. Moreover, the properties of $(\C, \A = (\A_i)_{i \in \mathbb{Z}})$ will typically imply that the assumptions of Theorem \ref{devissage-multiple-type} are satisfied by the Waldhausen category $\C_{\geq i}$, which consists of objects in $\C$ concentrated in degrees $\geq i$,  together with the collection of subcategories $\A_{\geq i} = (\A_j)_{j \geq i}$. 

\smallskip

For example, this happens in the context of the Gillet--Waldhausen theorem \cite[1.11.7]{TT}. In this case, $\C = \Ch(\mathcal{E})$ is the Waldhausen category of bounded chain complexes in an exact category $\mathcal{E}$. We assume that $\mathcal{E}$ arises as a full additive subcategory of an abelian category $\a$ which is closed under extensions and kernels of epimorphisms (cf. \cite[1.11.3--1.11.7]{TT}). By the Gabriel--Quillen embedding theorem, this is satisfied by every exact category which is weakly idempotent complete (see \cite[A.7.1 and A.7.16(b)]{TT}). Thus, we may make this assumption on $\mathcal{E}$ without loss of generality, since we may pass to the idempotent completion and use cofinality (see \cite[A.9.1]{TT}, \cite{Ci2}).  Then $\A_i$ is defined to be the full Waldhausen subcategory of bounded chain complexes whose homology (computed in $\Ch(\a)$) is in $\mathcal{E}$ and concentrated in degree $i$. Note that the inclusion $\mathcal{E} \hookrightarrow \A_i$ induces a $K$-equivalence $K(\mathcal{E}) \stackrel{\simeq}{\to} K(\A_i)$ where the homotopy inverse is induced by the exact functor $H_i \colon 
\A_i \to \mathcal{E}$. Moreover, if $\C_{\geq i} \subset \C$ denotes the Waldhausen subcategory of bounded chain complexes whose 
homology is concentrated in degrees $\geq i$, then the inclusion $\Ch(\mathcal{E})_{\geq i} \to \C_{\geq i}$ induces a $K$-equivalence. 

\smallskip

In these cases, after applying Theorem \ref{devissage-multiple-type},  we obtain homotopy equivalences $K(\A_i) \simeq K(\C_{\geq i})$ for any $i \in \mathbb{Z}$. Moreover, the exact inclusion functors $$\C_{\geq i+1} \to \C_{\geq i}, \ \ i \in \mathbb{Z},$$ will induce homotopy equivalences in $K$-theory by the Additivity Theorem -- this inclusion functor will typically be identified with an equivalence of homotopy theories $\C_{\geq i+1} \simeq \C_{\geq i}$ followed by the suspension in $\C_{\geq i}$. In this way we obtain homotopy equivalences: $K(\A_i) \simeq K(\C_{\geq i}) \simeq \mathrm{hocolim}_{i \in \mathbb{Z}} K(\C_{\geq i}) \simeq K(\C)$. 
 \end{example}
 
\begin{example}(`Theorem of the heart' for weight structures) \label{theorem_of_the_heart}
The setting of Theorem \ref{devissage-multiple-type} arises in the stable context naturally from a weight structure. Thus, Theorem \ref{devissage-multiple-type} can be used to conclude the theorem of the heart for weight structures in a general stable $\infty$-category that was shown in \cite{Fo}.  To see this, let $\C$ be a stable $\infty$-category with a bounded weight structure -- we will omit the details concerning the passage between good Waldhausen categories and their associated $\infty$-categories (see, for example, \cite[7.2]{BGT}). We denote by $\A_i$, $i \in \mathbb{Z}$, the full subcategory which consists of those objects concentrated in degree $i$ with respect to the weight structure. Then the assumptions of Theorem \ref{devissage-multiple-type} are basic consequences of the properties of weight structures. More specifically, the d\'evissage condition follows from the weight decomposition, by induction and using the fact that the weight structure is bounded (see \cite[Proposition 3.31]{Fo}). The cancellation property is essentially obvious in this context (cf. \cite[Proposition 3.10]{Fo}), as a consequence of the fact that a morphism $f \colon X \to Y$ is $k$-connected (with respect to $\A$) if and only if its cofiber is concentrated in degrees $\geq k + 1$ . Lastly, the admissibility property also follows easily from the properties of the weight structure (cf. \cite[Proposition 5.4]{Fo}, \cite[Lemma 1.7.4]{Wa}). Thus, Theorem \ref{devissage-multiple-type} and Example \ref{Gillet-Waldhausen} show that the inclusion of the heart, $\A_0 \hookrightarrow \C$, induces a homotopy equivalence $K(\A_0) \simeq K(\C)$. 

Using different methods, the theorem of the heart for weight structures is also obtained in the recent work of Hebestreit--Steimle \cite[8.1.3]{HS} as a consequence of a more general theorem of the heart in the context of Grothendieck-Witt theory \cite[Theorem B]{HS}.
\end{example}


\begin{thebibliography}{10}

\bibitem{ABG}
Benjamin~Antieau, Tobias~Barthel, and David~Gepner, \emph{On localization sequences in the algebraic K-theory of ring spectra}, Journal of the European Mathematical Society 20 (2018), no. 2, 459--487. 

\bibitem{AGH}
Benjamin~Antieau, David~Gepner, and Jeremiah~Heller, \emph{K-theoretic obstructions to bounded t-structures}, Invent. Math. 216 (2019), no. 1, 241--300. 

\bibitem{Bar}
Clark~Barwick, \emph{On exact $\infty$-categories and the Theorem of the Heart}, Compos. Math. 151 (2015), no. 11, 2160--2186.

\bibitem{BGT}
Andrew~J.~Blumberg, David~Gepner, and Gon\c{c}alo~Tabuada, \emph{A universal characterization of higher algebraic K-theory},  Geom. Topol. 17 (2013),  no. 2, 733--838.

\bibitem{BM}
Andrew~J.~Blumberg and Michael~A.~Mandell, \emph{The localization sequence for the algebraic K-theory of topological K-theory}, Acta Math. 200 (2008), no. 2, 155--179. 

\bibitem{BM2}
Andrew~J.~Blumberg and Michael~A.~Mandell, \emph{Algebraic K-theory and abstract homotopy theory}, Adv. Math. 226 (2011), no. 4, 
3760--3812. 

\bibitem{CZ}
Jonathan~A.~Campbell and Inna~Zakharevich, \emph{Devissage and Localization for the Grothendieck Spectrum of Varieties}, arXiv: \url{https://arxiv.org/abs/1811.08014}

\bibitem{Ci2}
Denis-Charles~Cisinski, \emph{Th\'{e}or$\grave{e}$mes de cofinalit\'{e} en K-Th\'{e}orie (d'apr$\grave{e}$s Thomason)}. Preprint (2002). \url{http://www.mathematik.uni-regensburg.de/cisinski/cofdev.pdf}

\bibitem{Ci}
Denis-Charles~Cisinski, \emph{Invariance de la K-th\'{e}orie par \'{e}quivalences d\'{e}riv\'{e}es}, J. K-Theory 6 (2010), no. 3, 505--546.

\bibitem{Fo}
Ernest~E.~Fontes, \emph{Weight structures and the algebraic K-theory of stable $\infty$-categories}, arXiv: \url{https://arxiv.org/abs/1812.09751}

\bibitem{HS}
Fabian~Hebestreit and Wolfgang~Steimle, \emph{Stable moduli spaces of hermitian forms} (with an appendix by Y. Harpaz), arXiv: \url{https://arxiv.org/abs/2103.13911}

\bibitem{Qu}
Daniel~Quillen, \emph{Higher algebraic K-theory. I.}, Algebraic K-theory, I: Higher K-theories (Proc. Conf., Battelle Memorial Inst., Seattle, Wash., 1972), 
pp. 85--147, Lecture Notes in Math. 341, Springer, Berlin 1973. 


\bibitem{Sc}
Marco~Schlichting, \emph{Negative $K$-theory of derived categories}, Math. Z. 253 (2006), no. 1, 97--134. 

\bibitem{TT}
R.~W.~Thomason and Thomas Trobaugh, \emph{Higher algebraic K-theory of schemes and of derived categories}, 
The Grothendieck Festschrift, Vol. III, pp. 247--435, Progr. Math., 88, Birkh\"auser Boston, Boston, MA, 1990. 

\bibitem{Wa2}
Friedhelm~Waldhausen, \emph{Algebraic K-theory of spaces, localization, and the chromatic filtration of stable homotopy}, 
Algebraic topology (Aarhus, 1982), pp. 173--195, Lecture Notes in Math. 1051, Springer, Berlin, 1984.

\bibitem{Wa}
Friedhelm~Waldhausen, \emph{Algebraic K-theory of spaces}, Algebraic and geometric topology (New Brunswick, N.J., 1983), 
pp. 318--419, Lecture Notes in Math. 1126, Springer, Berlin, 1985. 

\bibitem{Ya}
Dongyuan~Yao, \emph{A devissage theorem in Waldhausen K-theory}, J. Algebra 176 (1995), no. 3, 755--761.

\end{thebibliography}
\end{document}